\documentclass[10pt]{article}
\usepackage{amsmath, amsthm, amssymb, }
\usepackage{float,epsfig}
\usepackage{color}

\begin{document}

\newtheorem{theorem}{\bf Theorem}[section]
\newtheorem{proposition}[theorem]{\bf Proposition}
\newtheorem{definition}[theorem]{\bf Definition}
\newtheorem{corollary}[theorem]{\bf Corollary}
\newtheorem{example}[theorem]{\bf Example}
\newtheorem{exam}[theorem]{\bf Example}
\newtheorem{remark}[theorem]{\bf Remark}
\newtheorem{lemma}[theorem]{\bf Lemma}
\newcommand{\nrm}[1]{|\!|\!| {#1} |\!|\!|}

\newcommand{\calL}{{\mathcal L}}
\newcommand{\calX}{{\mathcal X}}
\newcommand{\calA}{{\mathcal A}}
\newcommand{\calB}{{\mathcal B}}
\newcommand{\calC}{{\mathcal C}}
\newcommand{\calK}{{\mathcal K}}
\newcommand{\C}{{\mathbb C}}
\newcommand{\R}{{\mathbb R}}
\renewcommand{\SS}{{\mathbb S}}
\newcommand{\LL}{{\mathbb L}}
\newcommand{\st}{{\star}}
\def\kernel{\mathop{\rm kernel}\nolimits}
\def\sigan{\mathop{\rm span}\nolimits}

\newcommand{\klasse}{{\boldsymbol \Delta}}

\newcommand{\ba}{\begin{array}}
\newcommand{\ea}{\end{array}}
\newcommand{\von}{\vskip 1ex}
\newcommand{\vone}{\vskip 2ex}
\newcommand{\vtwo}{\vskip 4ex}
\newcommand{\dm}[1]{ {\displaystyle{#1} } }

\newcommand{\be}{\begin{equation}}
\newcommand{\ee}{\end{equation}}
\newcommand{\beano}{\begin{eqnarray*}}
\newcommand{\eeano}{\end{eqnarray*}}
\newcommand{\inp}[2]{\langle {#1} ,\,{#2} \rangle}
\def\bmatrix#1{\left[ \begin{matrix} #1 \end{matrix} \right]}
\def \noin{\noindent}
\newcommand{\evenindex}{\Pi_e}

%\newcommand {\proof} {\par{\it Proof}. \ignorespaces}
%\newcommand {\eproof}
%      {\sigace
%        {\ \vbox{\hrule\hbox{\vrule height1.3ex\hskip0.8ex\vrule}\hrule}}
%        \par}

%%%%%%%%%%%%%%%%%%%%%%%%%%%%%%%%%%%%%%%%%%%%%%%%%%%%%%%%%%%%%%%%%%%%%%%%%%

\def \R{{\mathbb R}}
\def \C{{\mathbb C}}
\def \K{{\mathbb K}}
\def \H{{\mathbb H}}

\def \T{{\mathbb T}}
\def \Pb{\mathrm{P}}
\def \N{{\mathbb N}}
\def \Ib{\mathrm{I}}
\def \Ls{{\Lambda}_{m-1}}
\def \Gb{\mathrm{G}}
\def \Hb{\mathrm{H}}
\def \Lam{{\Lambda}}

\def \Qb{\mathrm{Q}}
\def \Rb{\mathrm{R}}
\def \Mb{\mathrm{M}}
\def \norm{\nrm{\cdot}\equiv \nrm{\cdot}}

\def \P{{\mathbb{P}}_m(\C^{n\times n})}
\def \A{{{\mathbb P}_1(\C^{n\times n})}}
\def \H{{\mathbb H}}
\def \L{{\mathbb L}}
\def \G{{\F_{\tt{H}}}}
\def \S{\mathbb{S}}
\def \s{\mathbb{s}}
\def \sigmin{\sigma_{\min}}
\def \elam{\Lambda_{\epsilon}}
\def \slam{\Lambda^{\S}_{\epsilon}}
\def \Ib{\mathrm{I}}
\def \Tb{\mathrm{T}}
\def \d{{\delta}}

\def \Lb{\mathrm{L}}
\def \N{{\mathbb N}}
\def \Ls{{\Lambda}_{m-1}}
\def \Gb{\mathrm{G}}
\def \Hb{\mathrm{H}}
\def \Delta{\triangle}
\def \Rar{\Rightarrow}
\def \p{{\mathsf{p}(\lam; v)}}

\def \D{{\mathbb D}}

\def \tr{\mathrm{Tr}}
\def \cond{\mathrm{cond}}
\def \lam{\lambda}
\def \sig{\sigma}
\def \sign{\mathrm{sign}}

\def \ep{\epsilon}
\def \diag{\mathrm{diag}}
\def \rev{\mathrm{rev}}
\def \vec{\mathrm{vec}}

\def \ham{\mathsf{Ham}}
\def \herm{\mathsf{Herm}}
\def \sym{\mathsf{sym}}
\def \odd{\mathsf{sym}}
\def \en{\mathrm{even}}
\def \rank{\mathrm{rank}}
\def \pf{{\bf Proof: }}
\def \dist{\mathrm{dist}}
\def \rar{\rightarrow}

\def \rank{\mathrm{rank}}
\def \pf{{\bf Proof: }}
\def \dist{\mathrm{dist}}
\def \Re{\mathsf{Re}}
\def \Im{\mathsf{Im}}
\def \re{\mathsf{re}}
\def \im{\mathsf{im}}

\def \sym{\mathsf{sym}}
\def \sksym{\mathsf{skew\mbox{-}sym}}
\def \odd{\mathrm{odd}}
\def \even{\mathrm{even}}
\def \herm{\mathsf{Herm}}
\def \skherm{\mathsf{skew\mbox{-}Herm}}
\def \str{\mathrm{ Struct}}
\def \eproof{$\blacksquare$}

\def \bS{{\bf S}}
\def \cA{{\cal A}}
\def \E{{\mathcal E}}
\def \X{{\mathcal X}}
\def \F{{\mathcal F}}
\def \cH{\mathcal{H}}
\def \cJ{\mathcal{J}}
\def \tr{\mathrm{Tr}}
\def \range{\mathrm{Range}}
\def \adj{\star}
%\newcommand {\proof} {\par{\it Proof}. \ignorespaces}
%\newcommand {\eproof}
    %  {\sigace
        %{\ \vbox{\hrule\hbox{\vrule height1.3ex\hskip0.8ex\vrule}\hrule}}
        %\par}

\def \pal{\mathrm{palindromic}}
\def \palpen{\mathrm{palindromic~~ pencil}}
\def \palpoly{\mathrm{palindromic~~ polynomial}}
\def \odd{\mathrm{odd}}
\def \even{\mathrm{even}}
%\def \herm{\mathrm{Hermitian}}
%\def \skherm{\mathrm{skew\mbox{-}Hermitian}}
%\def \str{\mathrm{ Struct}}

%%%%%%%%%%%%%%%%%%%%%%%%%%%%%%%%%%%%%%%%%%%%%%%%%%%%%%%%%%%%%%%%%%%%%%%%%%%%%%%%

\title{ Updating structured matrix pencils with no spillover effect on unmeasured spectral data and deflating pair}
\author{ Bibhas Adhikari\thanks{Corresponding author, Department of Mathematics,
IIT Kharagpur, India, E-mail:
bibhas@maths.iitkgp.ac.in } \,  Biswa Nath Datta\thanks{Department of Mathematical Sciences, Northern Illinois University, USA, E-mail:
profbiswa@yahoo.com } \, Tinku Ganai\thanks{Department of Mathematics,
IIT Kharagpur, India, E-mail:
tinkuganaimath@gmail.com }  \, Michael Karow\thanks{Department of Mathematics,
TU Berlin, Germany, E-mail: karow@math.tu-berlin.de. }  }
\date{}

\maketitle
\thispagestyle{empty}

{\small \noin{\bf Abstract.}
This paper is devoted to the study of perturbations of a matrix pencil, structured or unstructured, such that a perturbed pencil will reproduce a given deflating pair while maintaining the invariance of the complementary deflating pair. If the latter is unknown, it is referred to as no spillover updating. The specific structures considered in this paper include  symmetric, Hermitian, $\star$-even, $\star$-odd and $\star$-skew-Hamiltonian/Hamiltonian pencils. This study is motivated by the well-known  Finite Element Model Updating Problem in structural dynamics, where the given deflating pair represents a set of given eigenpairs and the complementary deflating pair represents the remaining larger set of eigenpairs. Analytical expressions of structure preserving no spillover updating are determined for deflating pairs of structured matrix pencils. Besides, parametric representations of all possible unstructured perturbations are obtained when the complementary deflating pair of a given unstructured pencil is known. In addition, parametric expressions are obtained for structured updating with certain desirable structures which relate to existing results on structure preservation of a symmetric positive definite or semi definite matrix pencil. 

 }

\vone \noin{\bf Keywords.} Model updating, structured matrix pencils, inverse eigenvalue problem, deflating subspace

\vone\noin{\bf AMS subject classifications.} 15A22, 65F18, 93B55, 46E30, 47A75

%\vone\noin {\bf AMS subject classification(2000):}  bnb

\section{Introduction}\label{sec:1}

The model updating problem (MUP) with no spillover effect on unmeasured spectral data has found its place in the core research areas of numerical linear algebra  due to its importance in real world applications, for example, in vibration industries including automobile, space and aircraft industries \cite{Dattabook,Carvalho2,Mottershedbook,Mottershedsurvey}. The problem is to update a   quadratic matrix polynomial in such a way that a small number of measured eigenvalues and eigenvectors are reproduced by the updated model while maintaining the no spillover of the large number of remaining unmeasured eigenpairs. It is of utmost practical interest that the finite-element inherited structures, such as the symmetry, positive definiteness or semi-definiteness are preserved in the updated model. The quadratic finite element model associated with the MUP is given by 
\be\label{qmodel} M\ddot x(t) + D\dot x(t) + Kx(t)=0\ee 
where $M, D, K$ are square matrices of dimension, say $n\times n,$ $x(t)$ is a column vector of order $n.$ Usually, $M$ is called mass matrix which is Hermitian positive definite, $K$ is Hermitian positive semi-definite and called stiffness matrix, and $D$ is a Hermitian matrix which is called the damping matrix \cite{Carvalho2,KuoDatta,Chu}. The equation (\ref{qmodel}) represents an undamped model if $D$ is the zero matrix. Solutions of (\ref{qmodel}) can be obtained as $x(t)=x_0e^{\lambda_0 t},$ where $(\lambda_0, x_0)$ turns out to be  eigenpairs of the quadratic matrix polynomial $Q(\lam)=\lam^2 M + \lam D + K\in\C^{n\times n}[\lam].$

 Let $\{(\lam_i, x_i) : i=1,\hdots, 2n\}$ be a collection of eigenpairs of $Q(z).$ Then given a positive integer $p\ll 2n$ and a set of scalars $\mu_i,  i=1,\hdots, p,$ the model updating problem is concerned with finding structure preserving quadratic matrix polynomials $\triangle Q(z)=\lam^2 \triangle M + \lam \triangle D  +\triangle K \in\C^{n\times n}[\lam]$ such that \be\label{q1}(Q(\mu_i) + \triangle Q (\mu_i))y_i =0, i=1,\hdots, p\ee for some $y_i\neq 0.$ In addition, if $(\lam_j, x_j), j=p+1,\hdots, 2n$  are not known then it is a no spillover updating. That is, \be\label{q2}(Q(\lam_j) + \triangle Q (\lam_j))x_j =0, j=p+1,\hdots, 2n\ee for such $\triangle Q(z)$ \cite{DattaSarkissian,Elhay}. In the context of applications, equation (\ref{qmodel}) represents a theoretical finite-element model of a structure that needs to be updated by a few measured eigenvalues ($\mu_i, i=1,\hdots, p$) obtained from the real structure without disturbing the unmeasured eigenvalues ($\lam_j, j=p+1,\hdots, 2n$) of the model. Several attempts have been made to solve the problem both by finding analytical and algorithmic solutions \cite{Baruch1,Baruch2,BermanNagy,Carvalho,Wei2,Zimmerman,Brahma,BaiDatta,QianXuBai,CaesarPeter,Carvalho2,ChuChu,ChuLin}. However, a complete characterization of solution sets describing $\triangle Q(z)$ which satisfy  (\ref{q1}) and (\ref{q2}) remains an open problem \cite{KuoDatta}.

We emphasize that a solution of the no spillover quadratic model updating does not necessarily yield a solution of the no spillover linear updating, just be setting the damped matrix to be the null matrix. 
For example: 
\begin{itemize} \item Consider the solution sets proposed in \cite{ChuChu} and \cite{ChuLin} for quadratic models. In  \cite{ChuChu}, $M$ is symmetric positive definite, $D$ is symmetric and $K$ is  symmetric positive definite, and in \cite{ChuLin}, the authors consider a same structure of $Q(\lam)$ but $K$ is semi-definite. Setting $D=0$ in the solutions proposed both in \cite{ChuChu} and \cite{ChuLin}, it can be seen that the perturbation $\Delta D$ is a nonzero matrix. Hence the proposed solutions do not solve the MUP with no spillover for undamped structural models. 

%On the other hand, setting $M=0$ the proposed solutions provide $\Delta M=0, \Delta D=0$ and $\Delta K$ is a nonzero matrix which is a symmetric. Hence the solution sets proposed in \cite{ChuChu} and  \cite{ChuLin} can be utilized to derive solutions of MUP with no spillover for matrix pencils of the form $L(\lam)=\lam M + K$ where $M$ is symmetric, $K$ is symmetric positive definite or semi-definite, and only the matrix $K$ is perturbed. 

%\item in \cite{ChuLin}, the authors proposed solution for MUP with no spillover for quadratic polynomials of the form $Q(\lam)=\lam^2 M + \lam D + K\in\R^{n\times n}[\lam]$ where $M$ is symmetric positive definite, $D$ is symmetric and $K$ is symmetric positive semi-definite. Here setting $D=0,$ the proposed solution provides $\Delta D\neq 0,$ and hence it can not provide solutions for pencils of the form $L(z)=z^2M+K.$ 

 \item In \cite{KuoDatta}, the authors consider quadratic models $Q(\lam),$ where $M$ is a real symmetric nonsingular matrix, $D$ and $K$ are symmetric matrices. However, it can be easily checked that setting $D=0,$ the proposed solution provides $\Delta D\neq 0.$

%However, setting $D=0$ it can be checked that the proposed solutions provide $\Delta D\neq 0.$ Hence these solutions can not be considered as solutions for undamped structural models.

 \item In \cite{Lancaster} and \cite{Lancater2}, the author considers the MUP problem with/without spillover for quadratic models where $M$ is symmetric/Hermitian positive deifinite, $D$ and $K$ are symmetric/Hermitian matrices. However the author utilizes the Jordan pair of $Q(\lam)$ in order to redefine the problem in terms of self-adjoint triple, and the coefficient matrices $M, D, K$ are written using the moments of the corresponding system. Due to this formulation, it is not clear how setting $D$ to be the zero matrix will produce structured perturbations of the linear pencil from the solution of quadratic model, unless the Jordan pair satisfies an orthogonality condition. 

   \end{itemize}

% are obtained in \cite{ChuChu,Lancater2,Chu,Lancaster}. However, such 

Thus it may be concluded that the MUP with/without spillover for quadratic models and undamped models are inherently different if $M$ is a positive definite matrix. In this paper we consider the MUP with no spillover for undamped models $M\ddot x(t) + Kx(t)=0$ represented by structured matrix pencils described as follows.

 For $A \in \C^{m \times n}$ let 
$A^T$ denote its transpose and let $A^*=\bar A^T$ denote its conjugate transpose.
Let $\st\in\{*,T\}$ and $\epsilon_1,\epsilon_2\in \{-1,1\}$.
We say that the pencil $L(\lambda)=\lambda\, M+K\in \C^{n\times n}[\lambda]$ has $(\st,\epsilon_1,\epsilon_2)$-structure if
\begin{equation}
 M^\st=\epsilon_1\, M,\qquad K^\st=\epsilon_2\, K.
 \end{equation}
 Pencils of this form 
are known under the following names.
\begin{center}
\begin{tabular}{l|l}
name & $(\st,\epsilon_1,\epsilon_2)$
\\
\hline
symmetric & $(T,1,1)$\\
Hermitian & $(*,1,1)$\\
$T$-odd & $(T,1,-1)$\\
$*$-odd & $(*,1,-1)$\\
$T$-even & $(T,-1,1)$ \\
$*$-even & $(*,-1,1)$  \\
\end{tabular}
\end{center}

 The set of these pencils is denoted by $\L_n{(\st,\epsilon_1,\epsilon_2)}$. We also consider $\st$-skew-Hamiltonian/Hamiltonian matrix pencils  $L(\lambda)=\lambda\, M+K\in \C^{2n\times 2n}[\lambda]$ which appear in different applications including gyroscopic systems and linear response theory, where $(JM)^\star=- JM,$ $(JK)^\star=JK$ and $J=\bmatrix{0 & I_n\\ -I_n & 0}$ \cite{benner2002numerical}. Thus $JL(\lam)\in \L_{2n}(\st,-1,1).$ These structured matrix pencils arise in a variety of real world problems, see \cite{MMMM,QEP}.

Now, we define MUP with no spillover effect on unmeasured spectral data for pencils $L(\lam)=\lam M + K$  as follows. \\\\
\indent
(\textbf{P1}) {\bf (model updating problem with no spillover)}
Let $(\lam_i^c, x_i^c), i=1,\hdots, p$ be a collection of given eigenpairs of $L(\lam).$ Suppose $(\lam_j^f, x_j^f), j=p+1, \hdots, n$ is a collection of complementary eigenpairs of $L(\lam),$ that is $\{x_1,\hdots, x_n\}$ is nonsingular. Let $\lam_i^a$ and $x_i^a$ be a collection of given scalars and nonzero vectors respectively, $i=1,\hdots,p.$ Then determine perturbations $(\triangle M, \triangle K)$ such that 
$(\lam_i^a,x_i^a)$ become eigenpairs of 
${L}_\triangle (\lam)=\lam (M+\triangle M)+ (K+ \triangle K),$ 
and the corresponding complementary eigenpairs of $L_\triangle(\lam)$ are given by $(\lam_j^f, x_j^f), j=p+1,\hdots, n.$ (The notations $^c, ^f, ^a$ stand for {\it change, fixed} and {\it aimed} respectively.)

Besides, determine $\triangle M, \triangle K$ such that 
 $L_\triangle (\lam) \in\S\subseteq \L_n{(\st,\epsilon_1,\epsilon_2)}$ 
 whenever $L(\lam)\in \S$ and $(\lam_j^f, x_j^f), j=p+1, \hdots, n$ are not known, where $\S$ is a set of structured matrix pencils.\\\\ 

Setting  $\Lambda_a=\diag\{\lam^a_i : i=1,\hdots, p\},$ $X_a=[x_1^a, \, x_2^a, \, \hdots, x_p^a],$ $\Lambda_f=\diag\{\lam_j^f : j=p+1,\hdots, n\},$ and $X_f=[x_{p+1}^f, \, x_{p+2}^f, \, \hdots, x_n^f],$ it follows from Problem (\textbf{P1}) that the desired perturbations $(\Delta M, \Delta K)$ should satisfy
$$(M+\Delta M)X_a\Lambda_a+(K+\Delta K)X_a=0, \qquad 
(M+\Delta M)X_f\Lambda_f+(K+\Delta K)X_f=0.
$$

The matrix pairs $(X,\Lambda)$ with $MX\Lambda+KX=0$ are called deflating pairs
 of $\lambda M+K$ \cite{controlbook}. Here it is not required that $\Lambda$ is to be diagonal.
 However, to avoid redundancies $X,$ should have full column rank.
 Two deflating pairs $(X_1,\Lambda_1)$, $(X_2,\Lambda_2)$ are said to be 
 complementary if $\bmatrix{X_1 & X_2}$  is a nonsingular square matrix.
With this terminology the following extended problem can be formulated. \\\\

({\bf P2}) {\bf (change of deflating pairs with no spillover)} Let  $(X_c,\Lambda_c)\in \C^{n\times p} \times \C^{p\times p}  $ and 
$(X_f,\Lambda_f)\in \C^{n\times (n-p)} \times \C^{(n-p)\times (n-p)} $ be 
complementary deflating pairs of a marix pencil $L(\lambda)=\lambda\, M+K.$
Let $(X_a,\Lambda_a)$ be a matrix pair of the same dimension 
as  $(X_c,\Lambda_c)$ such that $\bmatrix{X_a & X_f}$ is nonsingular.
 Find perturbations $(\Delta M,\Delta K)$
 such that $(X_a,\Lambda_a)$ and $(X_f,\Lambda_f)$
 are complementary deflating pairs of the perturbed pencil
$L_\Delta(\lambda)=(M+\Delta M)\, \lambda+ (K+\Delta K).$ 

Moreover, determine pair of structured perturbations $(\triangle M, \triangle K)$ such that 
 $L_\triangle (\lam) \in \S\subseteq \L_n(\st, \epsilon_1, \epsilon_2)$ 
 whenever $L(\lam)\in \S$ and $(X_f,\Lambda_f)$ is not known, where $\S$ is a set of structured matrix pencils
(Note that $\Lam_c, \Lam_a, \Lam_f$ need not be diagonal matrices).\\\\

Let us call the complementary deflating pairs $(X_c,\Lam_c)$ and $(X_f, \Lam_f)$ of a pencil $L(\lam)\in\C^{n\times n}[\lam]$ as {\it change} and {\it fixed} deflating pairs respectively. Then it follows that the Problem {\bf(P1)} is a special case of Problem {\bf(P2)}. 

%The pair $(X_a, \Lam_a)$ which has same dimension as $(X_c, \Lam_c)$ is refereed to as {\it aimed} pair of matrices. 

Problem ({\bf P1}) for Hermitian pencils defines the standard MUP with no spillover for an undamped model by setting $\lambda=z^2$. It is extensively studied in literature. See \cite{DaiMKupdate,YuanKsparse,Wei,Li,SarmadiKE,Xie} and the references therein. However, explicit parametric expressions of $\triangle M, \triangle K$ are obtained only in a few articles when both the coefficient matrices of $L(\lam)$ are positive definite or semi-definite. For example:
\begin{itemize}
\item[$\blacksquare$] In \cite{Carvalho2}, Carvalho et al. have derived solutions of problem {\bf (P1)} which are of the form $\triangle M=0, \triangle K=-MX_c\Psi X_c^TM$ for an undamped model $L(z^2)=z^2 M+K\in \R^{n\times n}$ where both $M$ and $K$ are symmetric positive definite, and $\{\lam_1^c, \hdots, \lam_p^c\}\cap \{\lam_{p+1}^f, \hdots, \lam_{2n}^f\}=\emptyset$. 
Here $\Psi$ is a (symmetric) solution of a (matrix) linear system, which has to obtained by solving the system numerically.

\item[$\blacksquare$] Solvability conditions and explicit expressions for solution pairs $(\triangle M, \triangle K)$ are obtained by Mao et al. in \cite{DaiMKupdate} for $L(\lam)=\lam M - K\in\R^{n\times n},$ where $M$ positive definite and $K$ is positive semi-definite. 

%$\{\lam_1^c, \hdots, \lam_p^c\}\cap \{\lam_{p+1}^f, \hdots, \lam_{n}^f\}=\emptyset.$

%\item[$\blacksquare$] Berman et al. have produced explicit expressions for $\triangle M, \triangle K$ in \cite{BermanNagy} for the MUP by treating it as an optimization problem for a given pair of mass and stiffness matrices.
\end{itemize}

Analytical expressions of the updating matrices are also obtained for undamped models in \cite{YuanZuoChen} and \cite{YuanDai} by treating the MUP as a residual minimization problem and matrix pencil nearness problem respectively. An optimization approach is also considered in \cite{BermanNagy} to obtain the updates. Determination of explicit expressions for updating matrices is motivated by the fact that it gives more suitable results than the same obtained by using iterative methods \cite{YangChen}. Particular classes of solutions are also obtained for specific structural undamped models \cite{YuanMKupdate,FindMKblock}. To the best of the knowledge of the authors, no explicit solution sets are available in literature for the undamped model when the corresponding matrix pencils are not Hermitian.

The contribution of this work are as follows. Let $L(\lam)=\lam M + K.$
\begin{enumerate}
\item First, a general expression is obtained for all possible unstructured perturbations which solves the Problem ({\bf P2}) when the fixed (unmeasured) deflating pair of the corresponding pencil is known. 

%Thus given a matrix pencil $L(\lam)$ and complementary deflating pairs $(X_c,\Lam_c), (X_f,\Lam_f)$ of $L(\lam),$ and a pair $(X_a,\Lam_a)$ we determine all possible perturbations $\Delta M, \Delta K$ such that $(X_c,\Lam_c), (X_f,\Lam_f)$ become complementary deflating pairs of $L_\Delta(\lam),$ when $\bmatrix{X_f & X_a}$ is nonsingular. 

\item Next, parametric expressions are determined for structure preserving perturbations which solve the Problem ({\bf P2})  when $ L(\lam)\in \L_n{(\st,\epsilon_1,\epsilon_2)}.$ In this case, the fixed (unmeasured) deflating pair of $L(\lam)$ is unknown, and $\sigma(\Lam_c) \cap \sigma(\ep_1\ep_2\Lam_f^\st)=\emptyset.$  

\item Finally, parametric solutions of the Problem ({\bf P2}) are obtained for especially structured pencils $L(\lam)\in \S \subset \L_n{(\st,\epsilon_1,\epsilon_2)}.$ The pencils $L(\lam)\in \S$ have the following structures: Hermitian pencils with $M$ positive definite, $\st$-odd pencils with $M$ positive definite; $\st$-even pencils with $K$ positive definite; and $\st$-skew-Hamiltonian/Hamiltonian matrix pencils $L(\lam),$ that is,  $JL(\lam)\in \L_{2n}(\st,-1,1).$ 

Moreover, parametric solution sets for the Problem ({\bf P1}) are obtained by utilizing the solutions of the Problem ({\bf P2}) when $L(\lam)\in \S$. It is also shown that the proposed solution realizes the solution obtained by Carvalho et al. in \cite{Carvalho2} as a special case (see Remark \ref{remark:carvelo}). Besides, the proposed solution also identifies the solution proposed by  Mao et al. in \cite{DaiMKupdate} (see Remark \ref{HermMK}). It is also to be noted in this context that our results can not be obtained as special cases of the existing structured preserving results of the quadratic FEM updating just by setting the damping matrix to be the null matrix.

The obtained results are supported with numerical examples.
\end{enumerate}

The paper is organized as follows. In the next two sections
we present elementary facts on deflating pairs and pencils with $(\st,\epsilon_1,\epsilon_2)$-structure. Though all these fact are known we give
some proofs for the convenience of the reader.
In Section \ref{sec:unstr} we discuss Problem ({\bf P2}) for unstructured perturbations.
We give a general solution formula provided for the case that $(X_f,\Lambda_f)$ is completely known.
The latter rarely happens in practical applications. However, for pencils with
$(\st,\epsilon_1,\epsilon_2)$-structure the complete knowledge of $(X_f,\Lambda_f)$
is not required for solving the problem. Instead, only
a certain spectral condition is needed.  This is the content of Section \ref{sec:str}
in which we present our main result. 
In the remaining sections we discuss special cases and show numerical examples.

{\bf Notation.} As usual, $\R$ and $\C$ denote the field of real and complex numbers respectively. $A \geq 0$ denotes that $A$ is a Hermitian positive semi-definite matrix, whereas $A>0$ denotes that $A$ is Hermitian positive definite. 
$\|X\|_F$ denotes Frobenius norm of a matrix $X$. $\C^{n\times n}[\lam]$ denotes the space of one parameter ($\lam$) matrix polynomials whose coefficients are complex matrices of order $n\times n.$ 
By $\sigma(\Lambda)$ we denote the spectrum (that is the multiset of eigenvalues) of 
$\Lambda$.  $\re(x)$ and $\im(x)$ denote the real and imaginary parts of a vector or scalar $x.$ Finally, $I_k$ denotes the identity matrix of order $k\times k.$

\section{Eigenpairs and deflating pairs}
A pencil $L(\lambda)=\lambda \, M+K\in \C^{n \times n}[\lambda]$ is said to  be regular if its characteristic polynomial $\chi(\lambda)={\rm det}(\lambda M+K)$
is not zero polynomial. In this paper we consider only regular pencils. 
The zeros of $\chi$ are called the finite eigenvalues of $L(\lambda)$. 
The pencil is said to have eigenvalue infinity if $M$ is singular.  
Let $\lambda_0\in \C$ be a finite eigenvalue. Then there exists a nonzero
eigenvector $x \in \C^n$ such that $\lambda_0Mx+Kx=0$. The pair $(\lambda_0,x)$ is called an eigenpair of $L(\lambda)$.
Recall from  the introduction that a matrix pair $(X,\Lambda)\in \C^{n \times p}\times \C^{p \times p}$ with 
${\rm rank}\, X=p\leq n$.
is said to be a deflating pair for the pencil $L(\lambda)$ if 
\begin{equation}
MX\Lambda+KX=0.
\end{equation}
The latter is equivalent to the equation
$ L(\lambda)X=MX(\lambda\, I-\Lambda).$
The range of $X$ is then called a deflating subspace.
If $p=1$ then $(\Lambda,X)$ is an  eigenpair of $L(\lambda)$. 
In general the eigenvalues of the square matrix $\Lambda$ form a subset of the set of  eigenvalues of $L(\lambda)$. More precisely, if $\xi$ is an eigenvector of $\Lambda$ to the eigenvalue $\lambda_0 \in \C$  (that is $\Lambda \xi=\lambda_0\xi$) then $(\lambda_0,X \xi)$ is an eigenpair of $L(\lambda)$.
 In particular, if $\Lambda$ is diagonal then the columns of $X$ are eigenvectors of $L(\lambda)$.
Furthermore, for any $\xi_0 \in \C^p$ the function
$x(t)=X\, e^{\Lambda t}\xi_0$ fulfills the differential equation
$M\, \dot x(t)+K\, x(t)=0$.
We say that two deflating pairs $(X,\Lambda)$, $(\hat X,\hat \Lambda)$ of $L(\lambda)$ 
are complementary if $\bmatrix{X,\hat X}$ is a nonsingular square matrix.
In this case  $(\bmatrix{X,\hat X},{\rm diag}(\Lambda, \hat \Lambda))$ is a deflating pair
and
$$ L(\lambda)=M\,\bmatrix{X &\hat X}
\left(\lambda\, I-{\rm diag}(\Lambda, \hat \Lambda)
\right)\bmatrix{X & \hat X}^{-1}.$$
If $(X,\Lambda)$ is a deflating pair then $(XZ,Z^{-1}\Lambda Z)$ is also a deflating pair for any nonsingular matrix $Z \in \C^{p\times p}$. The associated 
deflating subspaces coincide. A simple application 
of this fact is as follows. Suppose $M$ and $K$ are real matrices and
$(\lambda,x)$ is an eigenpair with nonreal $\lambda$. Then the conjugate pair 
$(\bar \lambda, \bar x)$ is also an eigenpair. Suppose that $M$ is nonsingular.
Then $\lambda \not = \bar \lambda$ implies that the vectors $x,\bar x$ 
are linearly independent and hence, the matrices 
$\Lambda={\rm diag}(\lambda, \bar\lambda),X=\bmatrix{x & \bar x}$ form a deflating pair. A real deflating pair $(X_r,\Lambda_r)$ with ${\rm range}\, X_r= {\rm range}\, X$ is
$$ X_r=XZ=\bmatrix{\re{(x)} & \im{(x)}},\quad
\Lambda_r=Z^{-1}\Lambda Z=\bmatrix{\re{(\lambda)} & \im{(\lambda)}\\ -\im{(\lambda)} & \re{(\lambda)}}, \quad
 \text{where}\quad
Z=\frac{1}{2}\bmatrix{ 1 & -i\\ 1 &i}.
$$
\section{Structured pencils}
Let $\st\in\{*,T\}$ and $\epsilon_1,\epsilon_2\in \{-1,1\}$.
Recall from the introduction that $L(\lambda)=\lambda\, M+K\in \C^{n,n}[\lambda]$ 
is said to have $(\st,\epsilon_1,\epsilon_2)$-structure if
\begin{equation}\label{MKsym} 
 M^\st=\epsilon_1\, M,\qquad K^\st=\epsilon_2\, K.
 \end{equation}
The set of these pencils is denoted by 
$\L_n{(\st,\epsilon_1,\epsilon_2)}$. The number $x_1^\st Mx_2\in \C$ is called the $M$-scalar product of the vectors $x_1,x_2$. For $z\in \C$ we define $z^\st=\bar z$ (the conjugate of $z$) if $\st=*$ and  $z^\st=z$ 
if $\st=T$. 
Then we have $x_2^\st Mx_1=\epsilon_1(x_1^\st Mx_2)^\st$. 
This yields 
$$x^\st Mx\begin{cases} \in \R &\text{if }(\st, \epsilon_1)=(*,1),\\
 \in i\R &\text{if }(\st, \epsilon_1)=(*,-1),\\
 =0&\text{if }(\st, \epsilon_1)=(T,-1).
\end{cases}
$$
In the first of these cases ($M$ Hermitian) the matrix $M$ is said to be positive definite if $x^*Mx>0$ for all $x \not =0$. 
If $x_1^\st Mx_2=0$ 
then the vectors $x_1,x_2$ are said to be $M$-orthogonal. For a matrix $X\in \C^{n \times p}$ with columns $x_i$ the associated $M$-Gramian is $G=X^\st MX=[x_i^\st Mx_j]\in \C^{p\times p}$. Obvoiusly, $G^\st=\epsilon_1 G$.

The proposition below lists elementary properties of pencils with $(\st,\epsilon_1,\epsilon_2)$-structure.

\begin{proposition}\label{orthprops}
Let 
$L(\lambda)=\lambda\, M+K\in\L_n{(\st,\epsilon_1,\epsilon_2)}$. Then
\begin{itemize}
\item[(i)] $\lambda_0\in \C$ is an eigenvalue of $L(\lambda)$
if and only if $\epsilon_1\epsilon_2\lambda_0^\st$ is an eigenvalue of $L(\lambda)$.
\end{itemize}
Let $X_j\in \C^{n,p_j}$, let
$G_{jk}=X_j^\st M X_k$ and $F_{jk}=X_j^\st K X_k$ for $j,k\in \{1,2\}$ . Then 
\begin{itemize}
\item[(ii)] the pencil
$
[X_1,X_2]^\st L(\lambda)[X_1,X_2]
=
\lambda\,
\bmatrix{ G_{11} & G_{12}\\ G_{21} & G_{22} }
+
\bmatrix{F_{11} & F_{12}\\F_{21} & F_{22}}
%\in \L_{p_1+p_2}{(\st,\epsilon_1,\epsilon_2)}.
$
has $(\st,\epsilon_1,\epsilon_2)$-structure.
In particular,
$G_{jk}^\st=\epsilon_{1}G_{kj}$, 
$F_{jk}^\st=\epsilon_{2}F_{kj}$ and
$\lambda\,G_{jj}+F_{ jj}\in \L_{p_j}{(\st,\epsilon_1,\epsilon_2)}$.
\end{itemize}
Suppose $(X_j,\Lambda_j)$, $j=1,2$ are deflating pairs of $L(\lambda)$. Then for $j,k\in \{1,2\}$,
\begin{itemize}
\item[(iii)] 
$G_{jk}\Lambda_k=-F_{jk}
=\epsilon_1\epsilon_2\,\Lambda_j^* G_{jk }$,
\item[(iv)] the spectral property $\sigma(\Lambda_k)\cap 
\sigma(\epsilon_1\epsilon_2\Lambda_j^\st)=\emptyset$ implies
$G_{jk}=F_{jk}=0$,
\item[(v)] if $\sigma(\Lambda_1)\cap 
\sigma(\epsilon_1\epsilon_2\Lambda_2^\st)=\emptyset$ then
$$
[X_1,X_2]^\st L(\lambda)[X_1,X_2]=
{\rm diag}(\lambda\, G_{11}-G_{11}\Lambda_1,
\lambda\, G_{22}-G_{22}\Lambda_2 ).
$$
In particular $G_{11}$ and $G_{22}$ are both nonsingular if
$(X_1,\Lambda_1)$ and $(X_2,\Lambda_2)$ are complementary and 
$M$ or $K$ is nonsingular.
\end{itemize}
\end{proposition}
\proof
The matrix $\lambda_0\, M+K$ is singular if an only if
the matrix $\epsilon_1\epsilon_2 \lambda_0^\st M^\st+K^\star=
\epsilon_2(\lambda_0\, M+K)^\st$ is singular. Thus, $(i)$ holds.
$(ii)$ is immediate from (\ref{MKsym}). Multiplying the
 relation $MX_k\Lambda_k+KX_k=0$ from the left with
 $X_j^\st$ yields the first identity of $(iii)$. The second identity then follows from $(ii)$. Reordering terms in $(iii)$ we get the Sylvester equation
 $G_{jk}\Lambda_k-\epsilon_1\epsilon_2\,\Lambda_j^* G_{jk }=0.$
 By an elementary result on Sylvester equations we have $G_{jk}=0$ 
if the matrices $\Lambda_k$ and 
$\epsilon_1\epsilon_2\,\Lambda_j^*$ have disjoint spectra. Hence, $(iv)$. $(v)$ is immediate from $(ii)$ and $(iv)$.
\eproof 
\\\\
The matrices $X_1$ and $X_2$ in Proposition \ref{orthprops} may be identical. In this 
case we obtain from statement $(iv)$ the following
corollary.
\begin{corollary}
Let $(X,\Lambda)$ be a deflating pair of 
$\lambda\, M+K\in\L_n{(\st,\epsilon_1,\epsilon_2)}$ such that
$\sigma(\Lambda)\cap 
\sigma(\epsilon_1\epsilon_2\Lambda^\st)=\emptyset$. Then 
$X^\st MX=X^\st KX=0$.
\end{corollary}
A further corollary of Proposition \ref{orthprops} is obtained
if $X_1,X_2$ are chosen to be column vectors.
\begin{corollary}
Let $(\lambda_1,x_1)$ and $(\lambda_2,x_2)$ be eigenpairs of
 $\lambda\, M+K\in\L_n{(\st,\epsilon_1,\epsilon_2)}$.
If $\lambda_2 \not=\epsilon_1\epsilon_2\, \lambda_1^*$ then $x_1^*Mx_2=x_1^*Kx_2=0$.
\end{corollary}
If $(\lambda_0,x)$ is an eigenpair of $\lambda\, M+K$ then by multiplying the
relation $(\lambda_0M+K)x=0$ from the left with $x^\st$ we get
\begin{equation}\label{eq:rayleighquotient}
\lambda_0=-x^\st Kx/x^\st Mx
\end{equation}
 provided that  $x^\st Mx\not=0$. The latter trivialy holds
if $\st=*$ and $M$ is Hermitian and positive definite. However, by the corollary
above we have $x^\st Mx=0$ whenever $\lambda_0\not = \epsilon_1\epsilon_2\lambda_0^\st$.
In this case we have the following statement which is immediate from the previous results in
this section.
\begin{corollary}
Let $(\lambda_0,x)$ be an eigenpair of $\lambda\, M+K\in\L_n{(\st,\epsilon_1,\epsilon_2)}$
such that $\lambda_0\not = \epsilon_1\epsilon_2\lambda_0^\st$.
By part $(i)$ of Proposition \ref{orthprops} there exists an eigenpair
$(\epsilon_1\epsilon_2\lambda_0^\st,\hat x)$. Set $X:=\bmatrix{x & \hat x}$, 
$g:=\hat x^\st Mx$. Then $(X,{\rm diag}( \lambda_0, \epsilon_1\epsilon_2\lambda_0^\st))$ is a
deflating pair of $L(\lambda)$, and
$$
X^*MX=\bmatrix{0 & \epsilon_1\, g^\st \\ g & 0}, \qquad
X^*KX=\bmatrix{0 & -\epsilon_2\,\lambda_0^\st\, g^\st \\ -\lambda_0\, g & 0}.
$$ 
By scaling of $x$ one can achieve that $g=1$ or $g=0$.
\end{corollary}
The identity (\ref{eq:rayleighquotient}) yields the following basic fact.
\begin{proposition}
 Let $M$ be Hermitian and positive definite.
Then all eigenvalues of $\lambda M+K$ are real if $K$ 
is Hermitian. They are all negative if $K$ is Hermitian and positive definite.
The eigenvalues are all purely imaginary or $0$ if $K$ is skew-Hermitian.
\end{proposition}
It is a well known fact that to a Hermitian pencil with positive definite $M$ 
there exists a basis $\{x_i,\,i=1, \ldots,n\}$ of eigenvectors such that $x_i^\st Mx_j=0$ for $i\not=j$.
The general eigenstructure of pencils with $(\st,\epsilon_1,\epsilon_2)$-symmetry is 
somehow involved and will not be discussed here. We refer to the literature
\cite{adhikari2009structured,GohLaRo,Thompson,lin1999canonical}.
The next proposition shows how to constuct a complementary
deflating pair to a given one.
\begin{proposition}
Let $(X_1,\Lambda_1)\in  \C^{n\times p} \times \C^{p\times p}$ be a deflating pair of 
$\lambda\, M+K\in\L_n{(\st,\epsilon_1,\epsilon_2)}$.
Suppose that $M$ and $G_1:=X_1^\st MX_1$ are both nonsingular.
Let $X\in \C^{n\times (n-p)}$ be such that $\bmatrix{X_1 &X}$ is nonsingular. Set
$ X_2:=X-X_1G_1^{-1}(X_1^\st MX).$ Then 
\begin{itemize}
\item[(i)]
$X_1^\st MX_2=X_1^\st KX_2=0$ and $G_2:=X_2^\st MX_2$ is nonsingular.
\item[(ii)] Set $\Lambda_2:=-G_2^{-1}(X_2^\st KX_2)$. 
Then $(X_2,\Lambda_2)$ is a deflating pair of 
$L(\lambda)$ which is complementary to $(X_1,\Lambda_1)$.
\end{itemize}
\end{proposition}
\proof
$(i)$ The identity $X_1^\st MX_2=0$ is easily verified.
The identity $X_1^\st KX_2=\epsilon_2(X_2^\st KX_1)^\st=0$  follows from 
$X_2^\st MX_1=\epsilon_1\,(X_1^\st MX_2)^\st=0$ by multiplying 
$MX_1\Lambda_1+KX_1=0$ with $X_2^\st$ from the left. The nonsingularity of $G_2$ follows from 
$\bmatrix{X_1 &X_2}^\st M\bmatrix{X_1 &X_2}={\rm diag}(G_1,G_2)$ and 
the nonsingularity of the matrices on the left hand side.
$(ii)$ The matrix $\bmatrix{X_1 &X_2}
=\bmatrix{X_1 &X}\bmatrix{I &-G_1^{-1}(X_1^\st MX)\\ 0 &I}$ is nonsingular. Thus, $X_2$ has full column rank. The results obtained so far imply that 
$\bmatrix{X_1 &X_2}^\st (MX_2\Lambda_2+K X_2)=0$. Thus,
$MX_2\Lambda_2+K X_2=0.$ 
\eproof
\section{Unstructured updates}\label{sec:unstr}
We now discuss the updating problem ({\bf P2}) for pencils without any prescribed structure.
By assumption $(X_f,\Lambda_f)$ and $(X_c,\Lambda_c)$
are complementary deflating pairs of $L(\lambda)=\lambda\, M+K$. Thus,
\begin{equation}\label{eq:defpairconds}
M X_f\Lambda_f+KX_f =0,\qquad
M X_c\Lambda_c+KX_c =0.
\end{equation}
Since $(X_f,\Lambda_f)$ and $(X_a,\Lambda_a)$ should be
 complementary deflating pairs of the updated pencil $L_\Delta(\lambda)=\lambda\, (M+\Delta M)+(K+\Delta K)$
the matrices $\Delta M,\Delta K$ we seek for should satisfy
\begin{equation}
\begin{array}{rcl}
(M+\Delta M)X_f\Lambda_f+(K+\Delta K)X_f &=&0,\\[.1 cm]
(M+\Delta M)X_a\Lambda_a+(K+\Delta K)X_a &=&0.
\end{array}
\end{equation}
Because of (\ref{eq:defpairconds}) an equivalent system of equations is
\begin{equation}\label{eq:DMDK2eq}
\Delta M X_f\Lambda_f+\Delta KX_f =0,
\qquad
\Delta M X_a\Lambda_a+\Delta KX_a =R_a, 
\end{equation}
where
\begin{equation}\label{eq:Rdef}
R_a:=-(MX_a\Lambda_a+KX_a)
=M(X_c\Lambda_c -X_a\Lambda_a)+K(X_c-X_a).
\end{equation}
Notice that 
\begin{equation}
R_a=MX_c(\Lambda_c-\Lambda_a)\quad \text{if}\quad X_a=X_c.
\end{equation}
Equations (\ref{eq:DMDK2eq}) can be written as
\begin{equation}\label{eq:DMDK1eq}
\underbrace{\bmatrix{\Delta M &\Delta K}}_{Y}
\underbrace{\bmatrix{X_f\Lambda_f  & X_a\Lambda_a \\
X_f & X_a}}_{A} \;=\;
\underbrace{\bmatrix{0 &R_a}}_{B}.
\end{equation}
According to a basic result on linear matrix equations the general
solution of (\ref{eq:DMDK1eq}) is
$$ Y=BA^\dagger+Z(I-AA^\dagger), \qquad Z\in \C^{n,2n}\text{ arbitrary},$$
where $A^\dagger=(A^*A)^{-1}A^*$ is the Moore-Penrose generalized inverse of $A$. Observe that in the present case $A^*A$ is indeed nonsingular since $A$ has full column rank. The latter holds
because $\bmatrix{X_f &X_a}$ is nonsingular by assumption.
Hence we have obtained a parametrization of all possible 
updates $Y=\bmatrix{\Delta M &\Delta K}$ that solve problem 
{\bf (P2)}. However, that the solution requires the knowledge of the matrix $A$ and hence the knowledge of $(X_f,\Lambda_f)$. This information is often not available in 
the applications. In the next section on structured pencils
we will derive updates whose construction only requires the knowledge of 
$(X_c,\Lambda_c)$ and a property  of the spectrum $\sigma(\Lambda_f)$ which is generically satisfied.

The theorem below provides a convenient subset of the general solution
set to Problem {\bf (P2)}. This theorem prepares the result on structured pencils in the next section.
\begin{theorem}\label{maintheo1}
Suppose that the assumptions of Problem {\bf (P2)} hold.
Let $U\in \C^{n \times p}$ be the unique matrix satisfying
$U^\st X_f=0$, and $U^\st X_a=I_p$.
$(\text{i.e. } U =([I,0]\,[X_a,X_f]^{-1})^\st)$,
 where $\st\in\{*,T\}$. Let $\widetilde M, \widetilde K \in \C^{n \times p}$ be such that
\begin{equation}\label{eq:basiceq1}
\widetilde M\,\Lambda_a+\widetilde K=R_a.
\end{equation}
Then the matrices 
$ \Delta M=\widetilde M  U^\st$ and $\Delta K=\widetilde K U^\st$
satisfy the requirements of problem {\bf (P2)}.
\end{theorem}
\begin{proof}
The proof is a straightforward verification using (\ref{eq:DMDK2eq}).
\end{proof}
Notice that to any $\widetilde M$ there is a unique 
$\widetilde K$ that solves (\ref{eq:basiceq1}), namely 
$\widetilde K=R_a-\widetilde M\Lambda_a$. This yields a 
parametrization of all solutions. Another parametrization is obtained as follows.
Equation (\ref{eq:basiceq1}) can be written in the form
$\bmatrix{\widetilde M & \widetilde K}\bmatrix{ \Lambda_a\\ I_p}=R_a.$
Thus, all its solutions are given (see \cite{AdhAlam}) via the Penrose inverse as
$$
\bmatrix{\widetilde M & \widetilde K} =
R_a\bmatrix{\Lam_a \\ I_p}^\dagger+
\bmatrix{Z_1 & Z_2} \left( \bmatrix{I_p & 0 \\ 0 & I_p} - \bmatrix{\Lam_a \\ I_p} \bmatrix{\Lam_a \\ I_p}^\dagger\right),
\quad Z_1,Z_2\in \C^{n\times p}\text{ arbitrary}.
$$
More explicitly, with the notation $H_a:=(\Lam_a^*\Lambda_a + I_p)^{-1}$,
\begin{equation}\label{eq:paramet}
\begin{array}{rcl}
 \widetilde M &=& R_a H_a\Lam_a^*+
 Z_1(I_p - \Lam_a H_a\Lam_a^*) - Z_2H_a\Lam_a^*, \\
  \widetilde K& =&
  R_aH_a- Z_1 \Lam_a H_a +Z_2 (I_p - H_a).
  \end{array}
 \end{equation}
\section{A general update result for pencils with symmetry}\label{sec:str}
We now discuss the updating problem ({\bf P2}) for pencils with 
$(\st,\epsilon_1,\epsilon_2)$-symmetry. 
The update method below only changes $\Lambda_c$ and fixes $X_c$ as well as $X_f$,
that is $X_a=X_c$.  For changing $X_a$ see the Remark \ref{rem:aftermain}.
The main requirement that makes our method work is
the spectral assumption $(a)$ in the theorem below.
\begin{theorem}\label{theo:stuctmain1}
Let $(X_c,\Lambda_c)$ and $(X_f,\Lambda_f)$ be complementary deflating pairs of the pencil 
$L(\lambda)=\lambda\, M+K\in \L_n{(\st,\epsilon_1,\epsilon_2)}$, where $\Lambda_c\in\C^{p,p}$. 
Suppose that
\begin{center}
$(a)\;$ 
$\sigma(\Lambda_c)\cap 
\sigma(\epsilon_1\epsilon_2\Lambda_f^\st)=\emptyset$ \quad and
\quad $(b)\;$ $G:=X_c^\st MX_c$ is nonsingular.
\end{center}
Let $\Lambda_a,\hat M,\hat K\in\C^{p,p}$ be such that
\begin{equation}\label{eq:basiceq2}
\hat M \Lambda_a+\hat K=G(\Lambda_c-\Lambda_a).
\end{equation}
Set
$$\Delta M:=U\hat MU^\st,\qquad \Delta K:=U\hat K U^\st,
\qquad \text{where }\quad U:=MX_cG^{-1}.$$
Then $(X_c,\Lambda_a)$ and $(X_f,\Lambda_f)$ are complementary deflating pairs of the pencil $L_\Delta(\lambda)=(M+\Delta M)\, \lambda+(K+\Delta K)$. Furthermore, 
$L_\Delta(\lambda)\in\L_n{(\st,\epsilon_1,\epsilon_2)}$ whenever 
$\lambda\,\hat M+\hat K   \in \L_p{(\st,\epsilon_1,\epsilon_2)} $. The latter holds if and only if 
$\lambda\, \hat M+(\hat M+G)\Lambda_a\in
 \L_p{(\st,\epsilon_1,\epsilon_2)}$.
\end{theorem}
\proof
Obviously, $X_c^\st U=I$. By the by part $(iv)$ of Proposition \ref{orthprops} and the spectral condition
$(a)$ we have $X_f^\st U=0$.
For $X_a=X_c$ the matrix $R_a$ from (\ref{eq:Rdef})
satisfies $R_a=MX_c(\Lambda_c-\Lambda_a)=UG
(\Lambda_c-\Lambda_a)$. Hence  (\ref{eq:basiceq2}) implies
$$(U\hat M)\Lambda_a+(U\hat K)=R_a.$$
Thus, the first statement of the theorem follows from
Theorem \ref{maintheo1}. The other statements are obvious.
\eproof

\begin{remark}\label{rem:aftermain}
\begin{itemize}
\item[$(i)$] If $X_c^\st K X_c$ is nonsingular then $G=X_c^\st M X_c$ is also nonsingular, and
the matrix $U$ in Theorem \ref{theo:stuctmain1} may  be written in terms of $K$ as 
 $U=KX_c(X_c^\st KX_c)^{-1}$. To see this, multiply $MX_c\Lambda_c+KX_c=0$ from the left with $X_c^\st$ and reorder terms so that $G\Lambda_c= -X_c^\st KX_c$. Thus $G$ and $\Lambda_c$ are nonsingular and $U=MX_cG^{-1}=-KX_c\Lambda_c^{-1}G^{-1}=KX_c(X_c^\st KX_c)^{-1}$.
\item[$(ii)$] For a given $\hat M$ there is a unique $\hat K$ that solves
(\ref{eq:basiceq2}), namely $\hat K=G(\Lambda_c-\Lambda_a)-\hat M\Lambda_a$. 
This yields a parameterization of all solution pairs 
$(\hat M,\hat K)$. Analogously to the formula
 (\ref{eq:paramet}) an alternative parameterization of all solutions of (\ref{eq:basiceq2}) is given by 
\begin{equation} \label{exprMcapKcap}
\begin{array}{rcl}
  \hat M &=&  G(\Lambda_c-\Lambda_a)H_a\Lam_a^*+
 Z_1(I_p - \Lam_aH_a\Lam_a^*) - Z_2H_a\Lam_a^*, \\
  \hat K& =&
  G(\Lambda_c-\Lambda_a)H_a- Z_1 \Lam_a H_a +Z_2 (I_p - H_a). 
  \end{array}
 \end{equation} 
 where $H_a=(\Lam_a^*\Lambda_a + I_p)^{-1}$ and $Z_1,Z_2\in \C^{p\times p}$ are arbitrary. Indeed note that the equation  (\ref{eq:basiceq2}) can be written as $$\bmatrix{\hat M & \hat K}\bmatrix{\Lam_a \\ I_p} = G(\Lam_c-\Lam_a)$$ which is a linear system of the form $AX=B,$ where $X$ is a  full rank matrix and $A$ is unknown. All such $A$ can be written as $A=BX^\dagger + Z(I-XX^\dagger)$ for any arbitrary matrix $Z$ of compatible dimension, where $X^\dagger$ denotes the pseudoinverse of $X$ if the pair $(X, B)$ satisfies $BX^\dagger X=B,$ see  \cite{AdhAlam}. Thus the expression given by (\ref{exprMcapKcap}) can be obtained. Further, it may be noted that structured solution of the equation (\ref{eq:basiceq2}) can be obtained by imposing structural conditions on the parameters $Z_1, Z_2.$

 \item[$(iii)$] Let $Z \in \C^{p\times p}$ be nonsingular. 
 %Let $(\hat M,\hat K)$
 %be solutions of the modified equation
 %$$
 %\hat M \Lambda_a+\hat K=G(\Lambda_c-Z^{-1}\Lambda_aZ).
 %$$
 %Then by Theorem \ref{theo:stuctmain1}, $(X_c, Z^{-1}\Lambda_aZ)$ is a deflating pair of the associated pencil 
 %$L_{\Delta}(\lambda)$. Thus $(X_cZ, \Lambda_a)$ is also a deflating pair of 
 %$L_{\Delta}(\lambda)$. 
 Let $(\hat M,\hat K)$
 be solutions of the modified equation
 $$
 \hat M (Z\Lambda_aZ^{-1})+\hat K=G(\Lambda_c-Z\Lambda_aZ^{-1}).
 $$
 Then by Theorem \ref{theo:stuctmain1}, $(X_c, Z\Lambda_aZ^{-1})$ is a deflating pair of the associated pencil 
 $L_{\Delta}(\lambda)$. Thus $(X_cZ, \Lambda_a)$ is also a deflating pair of 
 $L_{\Delta}(\lambda)$.
 \item[(iv)] If the spectrum of $\Lambda_c$ is closed with respect 
 to the $(\st,\epsilon_1,\epsilon_2)$-symmetry, that is 
 $\sigma(\Lambda_c)= \sigma(\epsilon_1\epsilon_2\Lambda_c^\st)$ then 
 the spectral condition $(a)$ is satisfied if the eigenvalues of $\Lambda_c$ are
 all different from the eigenvalues of $\Lambda_f$.
\end{itemize}
\end{remark}
The next theorem is about a simple subclass of perturbations.
\begin{theorem}
In the situation of Theorem \ref{theo:stuctmain1} let $L_\Delta(\lambda)$ be defined 
by  
$\hat M=t\, G$, $\hat K=G(\Lambda_c-(1+t)\Lambda_a)$
 for some $t\in \R$, that is
 \begin{eqnarray*}
 \Delta M &=& U\hat MU^\st \,=\, 
 t\, MX_c(X^\st MX_c)^{-1}X_c^\st M,\\
 \Delta K &=& U\hat KU^\st \,=\,
 MX_c(\Lambda_c-(1+t)\Lambda_a)(X_c^\st MX_c)^{-1}X_c^\st M.
 \end{eqnarray*} 
Then $(X_c,\Lambda_a)$ and $(X_f,\Lambda_f)$ are complementary deflating pairs of $L_\Delta(\lambda)$.
 Suppose that $\Lambda_a$ satisfies
 $G\Lambda_a=\epsilon_2(G\Lambda_a)^\st.$
Then $L_\Delta(\lambda)\in \L_n{(\st,\epsilon_1,\epsilon_2)}$.
\end{theorem}
\section{Updates for especially structured matrix pencils}
In this section we determine parametric updates which solve the problem \textbf{(P1)} for specific structured matrix pencils which are subsets of Hermitian, $\star$-odd, $\star$-even matrix pencils. 

\subsection{The Hermitian case with positive definite $M$}
Suppose  that $L(\lambda)=\lambda M+K \in \L_{\herm}$ with positive definite $M$. 
Then all  eigenvalues of $L(\lambda)$ are real and there exists a basis 
$x_1^c,\ldots,x_p^c,\, x_{p+1}^f,\ldots x_{n}^f$
of eigenvectors such that $L(\lambda_i^c)x^c_i=L(\lambda_i^f)x^f_i=0$.
By normalizing the eigenvectors (apply for Gram-Schmidt if some $\lambda_i^c$ coincide) we may assume
that $(x_i^c)^* Mx^c_j=0$ for $i\not=j$ and $(x_i^c)^* Mx_i^c=1.$ Thus, the $M$-Gramian
of the matrix
$X_c=[x^c_1 \;\ldots\; x^c_p]$ satisfies $G=X_c^*MX_c=I_p$.
Let $\Lambda_c={\rm diag}(\lambda^c_i)$, $\Lambda_f={\rm diag}(\lambda^f_i)$.
The spectral condition $(a)$ in Theorem \ref{theo:stuctmain1} reads
$$\{\lam_1^c, \hdots, \lam_p^c\}\cap \{\lam_{p+1}^f, \hdots, \lam_{n}^f\}=\emptyset.$$
If this condition is fulfilled the update matrices in Theorem \ref{theo:stuctmain1}
are \begin{equation} \label{DelMDelK_M+ve}
\Delta M=MX_c\,\hat M X_c^*M, \qquad
\Delta K=MX_c\,(\Lambda_c-\Lambda_a-\hat M\Lambda_a) X_c^*M.\end{equation}
Both matrices are Hermitian if $\hat M$ and $\Lambda_a$ are diagonal and real. If $M$ and $K$ are 
real matrices then $X_c$ can also be chosen to be real, and consequently the update matrices are real, too.

\begin{remark}\label{remark:carvelo}(Recovery of results in Carvalho et al. \cite{Carvalho})
If $\Lambda_a$ is a real diagonal matrix then choosing $\hat M=0,$ we obtain $\Delta M=0$ and $\Delta K=MX_c(\Lambda_c-\Lambda_a)X_c^*M$ from $(\ref{DelMDelK_M+ve})$. On the otherhand, putting $\hat K=0$ and assuming $\Lambda_a$ to be nonsingular, we achieve $\Delta K=0$ and $\Delta M=MX_c(\Lambda_c \Lambda_a^{-1}-I_p)X_c^*M$.  

Here we mention that when $M$ and $K$ are real symmetric positive definite matrices then the solution $\Delta M=0$ and $\Delta K=MX_c(\Lambda_c-\Lambda_a)X_c^TM$  realizes the solution obtained by Carvalho et al. in \cite{Carvalho2} for undampted models of the form $L(\lam)=\lam^2 M + K$. In addition, in their paper, the authors provide the solution where $\Delta K = MX_c\Psi X_c^TM$ and $\Psi$ has to be obtained by solving a matrix equation numerically. In contrast, the proposed solution here can be obtained directly by setting $\Psi = (\Lambda_c^2-\Lambda_a^2).$
\end{remark}
\begin{remark} \label{HermMK}(Recovery of results in Mao et al. \cite{DaiMKupdate})
If $\{\lam_1^c, \hdots, \lam_p^c\}\cap \{\lam_{p+1}^f, \hdots, \lam_{n}^f\}=\emptyset$ and $\Lambda_a$ is a real diagonal matrix then the Hermitian update matrices in Theorem \ref{theo:stuctmain1}, are given by $\Delta M=MX_c \hat MX_c^*M$ and $\Delta K=MX_c \hat KX_c^*M$ with \begin{equation}\label{eq:mao}
\hat M=H_a\left[(\Lambda_c-\Lambda_a)\Lambda_a+Z_1-Z_2\Lambda_a \right], \,\, \,\, \hat K=H_a\left[(\Lambda_c-\Lambda_a)-Z_1\Lambda_a+Z_2 \Lambda_a^2 \right]
\end{equation} where $H_a=(\Lambda_a^2+I_p)^{-1}$ and $Z_1,\,Z_2$ are arbitrary real diagonal matrices of compatible sizes.  

It is be noted that these solution sets identify the solutions given by  Mao et al. in \cite{DaiMKupdate}. The perturbations obtained in their paper are given by \beano \Delta M &=& MX_c(\Phi -\delta_{p+1}I_p)X_c^TM + (\delta_{p+1} -1) M \\ \Delta K &=& MX_c(\Phi\Lam_a - \delta_{p+1}\Lam_c)X_c^TM + (\delta_{p+1}-1)K \eeano where $\Phi$ is a symmetric positive definite matrix which satisfies $\Phi\Lam_a = \Lam_a \Phi,$ and $\delta_{p+1}>0$ is a real number. Setting $Z_1=H_a^-1(\Phi -I_p),$ $Z_2=\Lam_c - \Lam_a,$ the perturbations derived in this paper become $$\Delta M=MX_c\hat M X_c^TM, \, \Delta K= MX_c\hat K X_c^TM$$ which realizes Mao et al.s solution when $\delta_{p+1}=1,$ where $\hat M, \hat K$ are given by equaton (\ref{eq:mao}).
%(\ref{exprMcapKcap})$

Moreover, if the diagonal matrices $Z_1$ and $Z_2$ are chosen such that $(\Lambda_c-\Lambda_a)\Lambda_a+Z_1-Z_2\Lambda_a$ is a diagonal matrix with non-negative diagonal entries then $\Delta M$ is a positive semi-definite matrix, that is $M+\Delta M>0$.
\end{remark}

\begin{corollary} \label{Cor:Herm1}
Let the conditions of Remark \ref{HermMK} be satisfied. Besides, assume that $K>0$ and $\lam_i^a<0, i=1,\hdots,p.$ Then if $Z_1=\mbox{diag}\left(z_{11}^{(1)},\dots ,z_{pp}^{(1)}\right)$ and $Z_2=\mbox{diag}\left(z_{11}^{(2)},\dots ,z_{pp}^{(2)}\right)$ are chosen such that $$z_{ii}^{(1)}-z_{ii}^{(2)} \lam_i^a \geq \, \max\, \left\{(\lambda^a_i-\lambda^c_i) \lambda_i^a, (\lambda^c_i/\lambda^a_i-1) \right\}, i=1,\dots,p$$ then the perturbations $\triangle M, \triangle K$ in Remark \ref{HermMK} are positive semi-definite matrices.
\end{corollary}
\begin{proof} Note that $\lam_i^c <0$ since $\lam_i^c = -\frac{(x^c_i)^*K x^c_i}{(x^c_i)^*M x^c_i}.$ Then the proof is straightforward and easy to check.
\end{proof}

In the following we explain how the above results can be used to solve the standard model updating problem with no spillover effect for undamped models. Suppose that $M>0$ and $K^*=K$ are complex matrices of order $n.$ Then the eigenvalues of the matrix pencil $L(\lam)=\lam^2 M+K$ occur in pair $(\lam,\,-\lam)$ corresponding to an eigenvector $x\in\C^n.$ Besides, $\lam$ is either a real number or a purely imaginary number. 

Let $(\pm \lam^c_i, x_i^c), i=1,\hdots, p$ denote the eigenpairs of $L(\lam)$ that are to be changed to the aimed eigenvalues $\pm \lam^a_i, i=1,\hdots,p$ of $L_\triangle(\lam)=\lam^2 (M+\triangle M)+(K+\triangle K),$ for some positive semi-definite Hermitian matrix $\triangle M$ and $\triangle K\in\H_n.$ Setting $\Lam_c=\diag\left(\lam_1^c,\lam_2^c,\,\dots\,,\lam_p^c\right)^2,$ $\Lam_a=\diag\left(\lam_1^a,\lam_2^a,\,\dots\,,\lam_p^a\right)^2,$ $\Lam_f=\diag\left(\lam_{p+1}^f,\lam_{p+2}^f,\,\dots\,,\lam_{n}^f\right)^2,$ and $X_c=\bmatrix{x_1^c & x_2^c & \dots & x_p^c},$ the MUP with no spillover effect for $L(\lam)$ translates to the problem {\bf (P1).}

We depict the same in the following example which is taken from \cite{Carvalho2}.

\begin{example} This example has been taken from \cite{Carvalho2}.
Suppose $L(\lam)=\lam^2 M+K$ with 
$M=\mbox{diag}\left(1.294,\,1.294,\,1.294,\,1.294,\,1.294 \right)>0$ and $$K=\bmatrix{
1188.5000& 196.6000& 0& 0& -642.4000 \\
196.6000& 626.3000& 0& -555.6000& 0 \\
0& 0& 1188.5000& -196.6000& -546.1000 \\
0& -555.6000& -196.6000& 626.3000& 196.6000 \\
-642.4000& 0& -546.1000& 196.6000& 4019.1000} >0.$$

Let $\lam^c_1=57.4206i,\,\lam^c_2=4.8629i$ and $\lam^a_1=57.4247i,\,\lam^a_2=4.8112i.$ Suppose that we want to replace the set of eigenvalues $\{\lam^c_1,\,-\lam^c_1,\,\lam^c_2,\,-\lam^c_2\}$ of $L(\lam)$ by the desired set of eigenvalues $\{\lam^a_1,\,-\lam^a_1,\,\lam^a_2,\,-\lam^a_2\}$ respectively. Thus
$\Lambda_c=\mbox{diag}\left(-3297.13,\,-23.648\right),\,\Lam_a=\mbox{diag}\left(-3297.6,\,-23.148\right)$ and \begin{center}$X_c=\bmatrix{
 -0.177539&   0.125286 \\
  -0.018246&  -0.611759 \\
  -0.153557&  -0.085635 \\
   0.056719&  -0.611579 \\
   0.845073&   0.038600}.$ \end{center} 

Then by Corollary \ref{Cor:Herm1}, choosing $Z_1=\mbox{diag}\left(0,\,0.021592\right),\,Z_2=\mbox{diag}\left( 0.47136,\, 0 \right)$ we obtain \begin{center} $\triangle M=10^{-3} \bmatrix{
0.5674& -2.7703& -0.3878& -2.7695& 0.1747\\
-2.7703& 13.5270& 1.8935& 13.5231& -0.8535 \\
-0.3878& 1.8935& 0.2651& 1.8930& -0.1196 \\
-2.7695& 13.5231& 1.8930& 13.5191& -0.8532 \\
0.1747& -0.8535& -0.1196& -0.8532& 0.0543
} \geq 0$ and \\ $\triangle K=10^{-2} \bmatrix{
2.4878& 0.2557& 2.1517& -0.7948& -11.8415 \\
0.2557& 0.0263& 0.2211& -0.0817& -1.2170 \\
2.1517& 0.2211& 1.8611& -0.6874& -10.2420 \\
-0.7948& -0.0817& -0.6874& 0.2539& 3.7831 \\
 -11.8415& -1.2170& -10.2420& 3.7831& 56.3650
 } \geq 0. $ \end{center}
On taking $\Lambda_f=\mbox{diag}\left(-679.39,\,-942.69,\,-968.03\right)$ and $X_f=\bmatrix{
0.547227&   0.642402&  -0.115946 \\
-0.262485&   0.244128&  -0.519345 \\
0.522356&  -0.545451&  -0.414139 \\
0.313086&  -0.033487&   0.544433 \\
0.183201&   0.043366&  -0.147365}$ we obtain
$\|(M+\triangle M)X_f \Lambda_f+(K+\triangle K)X_f \|_F=7.7524 \times 10^{-13}\,$ which shows that the unmeasured spectral data remain undisturbed. 

Hence we conclude that eigenvalues of $L_{\triangle}(\lam) = \lam^2(M+\triangle M)+(K+\triangle K)$ are $\{\lam^a_1,\,-\lam^a_1,\,\lam^a_2,\,-\lam^a_2\}.$ Therefore eigenvalues of $L(\lam)$ are replaced by the desired eigenvalues with maintaining no spillover condition.
\end{example}

\subsection{The $\star$-odd matrix pencils with positive definite $M$}
Suppose  that $L(\lambda)=\lambda M+K$ is a $*$-odd matrix pencil with positive definite $M$. 
Then all eigenvalues of $L(\lambda)$ are either zero or purely imaginary number and there exists a basis 
$x_1^c,\ldots,x_p^c,\, x_{p+1}^f,\ldots , x_{n}^f$
of eigenvectors such that $L(\lambda_i^c)x^c_i=L(\lambda_i^f)x^f_i=0$.
By normalizing the eigenvectors, we may assume
that $(x_i^c)^* Mx^c_j=0$ for $i\not=j$ and $(x_i^c)^* Mx_i^c=1.$ Then the $M$-Gramian is given by $G=X_c^*MX_c=I_p$ where
$X_c=[x^c_1 \;\ldots\; x^c_p]$.
Let $\Lambda_c={\rm diag}(\lambda^c_i)$, $\Lambda_f={\rm diag}(\lambda^f_i)$.
Assuming the spectral condition $(a)$ in Theorem \ref{theo:stuctmain1}, let
$$\{\lam_1^c, \hdots, \lam_p^c\}\cap \{\lam_{p+1}^f, \hdots, \lam_{n}^f\}=\emptyset.$$
Then the update matrices in Theorem \ref{theo:stuctmain1}
are \begin{equation} \label{DelMDelK_*odd}
\Delta M=MX_c\,\hat M X_c^*M, \qquad
\Delta K=MX_c\,(\Lambda_c-\Lambda_a-\hat M\Lambda_a) X_c^*M.\end{equation}
Thus $\Delta M=(\Delta M)^*, \Delta K=-(\Delta K)^*$ if $\hat M$ is a real diagonal matrix and $\Lambda_a$ is a imaginary diagonal matrix. 

%If $M$ and $K$ are real matrices then $X_c$ can also be chosen to be real, and consequently the update matrices are real, too.

\begin{remark}
If $\hat M$ in $(\ref{DelMDelK_*odd})$ is chosen to be a diagonal matrix with non-negative entries then $M+\Delta M >0.$
\end{remark}
\begin{remark} \label{*oddMK}
If $\{\lam_1^c, \hdots, \lam_p^c\}\cap \{\lam_{p+1}^f, \hdots, \lam_{n}^f\}=\emptyset$ and $\Lambda_a$ is a diagonal matrix with purely imaginary complex numbers, the structured update matrices in Theorem \ref{theo:stuctmain1} are given by $\Delta M=MX_c \hat MX_c^*M$ and $\Delta K=MX_c \hat KX_c^*M$ with \begin{equation}
\hat M=H_a\left[(\Lambda_a-\Lambda_c)\Lambda_a+Z_1+Z_2\Lambda_a \right], \,\, \,\, \hat K=H_a\left[(\Lambda_c-\Lambda_a)-Z_1\Lambda_a-Z_2 \Lambda_a^2 \right]
\end{equation} where $H_a=(I_p-\Lambda_a^2)^{-1},$ $Z_1$ is an arbitrary real diagonal matrix and $Z_2$ is an arbitrary diagonal matrix with purely imaginary diagonal entries.  

Moreover, if the diagonal matrices $Z_1$ and $Z_2$ are chosen such that $(\Lambda_a-\Lambda_c)\Lambda_a+Z_1+Z_2\Lambda_a$ is a diagonal matrix with non-negative diagonal entries then $\Delta M$ is a positive semi-definite matrix, that is $M+\Delta M>0.$
\end{remark}

Now let us consider $T$-odd matrix pencils $L(\lam)=\lam M+K\in \R^{n\times n}[\lam].$ Then obviously the complex eigenvalues of $L(\lam)$ are purely imaginary which exist in conjugate pairs, whereas zero can be the only real eigenvalue of $L(\lam).$ Moreover if $x$ is an eigenvector corresponding to the complex eigenvalue $\lam$ then $\overline{x}$ is an eigenvector corresponding to the eigenvalue $\overline{\lam}=-\lam$ of $L(\lam).$ As usual, let the nonzero eigenvalues $\lam^c_i, \overline{\lam_i^c}, 1\leq i \leq p$ of $L(\lam)$ are to be replaced by $\lam^a_i, \overline{\lam_i^a}$ with no spillover effect in the (structured) perturbed pencil $L_\Delta(\lam).$ If $x_i^c$ denotes the normalized (complex) eigenvector corresponding to the eigenvalue $\lam_i^c, 1\leq i\leq p$ then we may assume that $X^*_cMX_c=2I_{2p}$ where $X_c=[x_1^c \, \overline{x_1^c} \, \hdots \, x_p^c \, \overline{x_p^c}].$ This implies $\hat{X}_c^TM\hat{X}_c=I_{2p}$ where $\hat{X}_c=[\re(x_1^c) \, \im(x_1^c) \, \hdots \, \re(x_p^c) \, \im(x_p^c)].$ Indeed, note that $\hat{X}_c=X_cZ$ where $$Z=\diag\left(\frac{1}{2}\bmatrix{1 & -i \\ 1 & i}, \,\, \hdots, \,\, \frac{1}{2}\bmatrix{1 & -i \\ 1 & i}\right)\in\C^{2p\times 2p}.$$

Let $\Lambda_c=\mbox{diag}(\Lambda^c_1,\dots, \Lambda^c_p)$ and $\Lambda_a=\mbox{diag}(\Lambda^a_1,\dots, \Lambda^a_p)$ where $\Lambda^c_j=\bmatrix{0& \im(\lambda^c_j) \\ -\im(\lambda^c_j)& 0}$ and $\Lambda^a_j=\bmatrix{0& \im(\lambda^a_j) \\ -\im(\lambda^a_j)& 0}.$  If the condition $(a)$ of Theorem \ref{theo:stuctmain1} is met, then the update matrices are $$\Delta M=M\hat{X}_c\hat{M}\hat{X}_c^TM,\,\,\,\Delta K=M\hat{X}_c\hat{K}\hat{X}_c^T M$$ where $\hat{M}$ and $\hat{K}$ are solutions of equation (\ref{eq:basiceq2}) in which $X_c$ is replaced by $\hat{X}_c.$

Moreover setting $\hat{M}=\mbox{diag}(\alpha_1I_2,\dots ,\alpha_pI_2)$ $\alpha_1,\hdots,\alpha_p \in \R$ and $\hat{K}=\Lambda_c-\Lambda_a-\hat{M}\Lambda_a$ we obtain $\Delta M=\Delta M^T$ and $\Delta K=-\Delta K^T.$

%If the condition $(a)$ of Theorem \ref{theo:stuctmain1} is met, define the update matrices are $$\Delta M=M\hat{X}_c\hat{M}\hat{X}_c^TM,\,\,\,\Delta K=M\hat{X}_c(\Lambda_c-\Lambda_a-\hat{M}\Lambda_a)\hat{X}_c^T M$$ which solves the problem {\bf (P1)} where $\hat{M}=\mbox{diag}(\alpha_1I_2,\dots ,\alpha_pI_2)$ $\alpha_1,\hdots,\alpha_p \in \R$ are arbitrary. Note also thet $\Delta M=\Delta M^T$ and $\Delta K=-\Delta K^T.$ 

\begin{remark}
If the spectral condition $(a)$ in Theorem \ref{theo:stuctmain1} is met, then the structured update matrices are given by $\Delta M=M\hat{X}_c \hat M\hat{X}_c^TM$ and $\Delta K=M\hat{X}_c \hat K\hat{X}_c^TM$ with \begin{equation}
\hat M=H_a\left[(\Lambda_a-\Lambda_c)\Lambda_a+Z_1+Z_2\Lambda_a \right], \,\, \,\, \hat K=H_a\left[(\Lambda_c-\Lambda_a)-Z_1\Lambda_a-Z_2 \Lambda_a^2 \right]
\end{equation} where $H_a=(I_{2p}-\Lambda_a^2)^{-1}$ and $Z_k=\mbox{diag}(Z^{(k)}_{11},\dots ,Z^{(k)}_{pp}),\,k=1,2$ with $Z^{(1)}_{jj}=\alpha_j I_2$ and $Z^{(2)}_{jj}=\bmatrix{0& \beta_j\\-\beta_j& 0},\, \alpha_j,\,\beta_j \in \R.$ Thus $L_\Delta(\lam)=\lam(M+\Delta M)+ (K+\Delta K)$ is a real $T$-odd pencil.

Moreover, if the matrices $Z_1$ and $Z_2$ are chosen such that $(\Lambda_a-\Lambda_c)\Lambda_a+Z_1+Z_2\Lambda_a$ is a diagonal matrix with non-negative diagonal entries then $\Delta M$ is a positive semi-definite matrix, that is $M+\Delta M>0.$
\end{remark}

Now we consider an example to obtain solution of {\bf (P1)} for undamped models $L(\lam)=\lam^2 M+K$ with $M>0$ and $K^*=-K$, by utilizing Remark \ref{*oddMK}. The values of $\lam^2$ to satisfy $\det(\lam^2M+K)=0$ are either zero or purely imaginary numbers (not necessary to have self conjugate pair), that is, either $\lam=\pm \sqrt{(a/2)}(1+i)$ or $\lam=\pm \sqrt{(a/2)}(1-i)$ for some $a\geq 0$. So, here we define the set $\mathcal{E}=\left\{\pm \sqrt{(a/2)} (1+i),\,\,\pm \sqrt{(a/2)} (1-i)\,\,:\,\,a \geq 0 \right\}.$ Then the eigenvalues of the matrix pencil $L(\lam) = \lam^2 M +K$ occur in pair $(\lam, -\lam)$ corresponding to an eigenvector $x\in \C^n$ for some $\lam \in \mathcal{E}$.  

Let $(\pm \lam^c_i,x^c_i),\,i = 1,\dots,p$ denote the eigenpairs of
$L(\lam)=\lam^2 M+K$ and $\pm \lam_i^c$ are to be changed to the aimed eigenvalues $\pm \lam^a_i,\,i = 1,\dots,p$ of $L_{\triangle}(\lam) = \lam^ 2 (M + \triangle M) + (K + \triangle K)$, for some positive semi-definite matrix $\triangle M$ and $\triangle K =-(\Delta K)^*$ without spillover effect. 
Setting $\Lambda_c=\mbox{diag}\left(\lam^c_1,\lam^c_2,\,\dots\,,\lam^c_p\right)^2,\,\Lambda_a=\mbox{diag}\left(\lam^a_1,\lam^a_2,\,\dots\,,\lam^a_p\right)^2, \, \Lambda_f=\mbox{diag}\left(\lam^f_{p+1},\lam^f_{p+2},\,\dots\,,\lam^f_{n}\right)^2$ and $X_c=\left[x^c_1\,x^c_2\,\dots\,x^c_p\right],$ the MUP with no spillover effect for $L(\lam)$ translates to the problem {\bf{(P1)}}.

We consider the following example.

\begin{example}
Suppose $L(\lam)=\lam^2 M+K$ with \begin{center}
$M=\footnotesize \bmatrix{
7.73863+ 0.00000i &  -1.98637 - 4.01069i&   4.09960 - 3.39198i&  -0.13418 + 2.89422i \\
  -1.98637 + 4.01069i&   6.55893+ 0.00000i &   1.90812 + 3.90598i&  -2.03549 + 1.81182i \\
   4.09960 + 3.39198i&   1.90812 - 3.90598i&   6.65654+ 0.00000i &   1.02186 + 1.42954i \\
  -0.13418 - 2.89422i&  -2.03549 - 1.81182i&   1.02186 - 1.42954i&   6.46526+ 0.00000i } >0,$\\ $K=\small \bmatrix{0.00000 + 3.90061i& 2.0140 - 0.30415i& 1.34863 + 1.79442i& 0.05369 - 1.38714i\\
-2.0140 - 0.30415i& 0.00000 - 2.49371i& 0.30279 + 1.11588i& 0.35925 - 1.54051i \\
-1.34863 + 1.79442i& -0.30279 + 1.11588i& 0.00000 - 0.49211i& -0.97818 - 1.32790i \\
-0.05369 - 1.38714i& -0.35925 - 1.54051i& 0.97818 - 1.32790i& 0.00000 + 1.85364i}.$
\end{center}
Let $\lam^c_1=1.30078(1+i),\,\lam^c_2=0.80933(1-i)$ and $\lam^a_1=0.82134(1-i),\,\lam^a_2=0.56214(1+i).$ Thus we want to replace the eigenvalues $\lam^c_1,\,-\lam^c_1,\,\lam^c_2,\,-\lam^c_2$ of $L(\lam)$ by the desired eigenvalues $\lam^a_1,\,-\lam^a_1,\,\lam^a_2,\,-\lam^a_2$ respectively. So we form 
$\Lambda_c=\mbox{diag}\left(3.3841i,\,- 1.3100i\right),\,\Lambda_a=\mbox{diag}\left(- 1.3492i,\,0.6320i\right)$ and \begin{center}
$X_c=\bmatrix{0.776569 - 0.000000i&   0.617954 - 0.000000i \\
   0.747129 - 0.098152i&   0.153552 + 0.005888i \\
  -0.714782 - 0.126987i&  -0.229136 - 0.266691i \\
   0.444742 + 0.301815i&   0.038972 + 0.500083i}.$ \end{center} 

Therefore by setting $Z_1=\mbox{diag}\left(8.9752,\,2.5715\right)$ and $Z_2=\mbox{diag}\left(-0.00717i,\,-0.60271i\right)$ we obtain \begin{center}
$\footnotesize \triangle M= \bmatrix{
2.91691 - 0.00000i&  -1.34898 + 0.69543i&   0.58908 - 1.38017i&  -2.65147 - 0.99875i \\
  -1.34898 - 0.69543i&   1.59117 + 0.00000i&  -0.77640 + 0.88115i&   1.14417 + 1.21855i \\
   0.58908 + 1.38017i&  -0.77640 - 0.88115i&   0.99350 + 0.00000i&  -0.03741 - 1.55808i \\
  -2.65147 + 0.99875i&   1.14417 - 1.21855i&  -0.03741 + 1.55808i&   2.80188 + 0.00000i} \geq 0,$ \\$\footnotesize\triangle K= \bmatrix{
0.00000 - 5.25520i&  -0.87564 + 0.59483i&  -1.14421 - 1.67866i&  -1.92869 + 4.08890i \\
   0.87564 + 0.59483i&   0.00000 + 6.19427i&  -2.65507 - 1.39903i & -1.45840 + 0.46343i \\
   1.14421 - 1.67866i&   2.65507 - 1.39903i&   0.00000 + 0.98530i&  -0.69246 + 1.08993i \\
   1.92869 + 4.08890i&   1.45840 + 0.46343i&   0.69246 + 1.08993i&   0.00000 - 3.49173i }.$  
\end{center}    

Taking $\Lambda_f=\mbox{diag}(-0.28296i,\,0.42255i)$ and $X_f=\bmatrix{
0.196502 + 0.024767i&  -0.048688 + 0.190081i \\
  -0.036828 + 0.054982i&   0.095288 - 0.254723i \\
  -0.150920 + 0.086267i&   0.466261 + 0.000000i \\
   0.231864 + 0.000000i&   0.099775 + 0.083410i}$ we obtain
$\|(M+\triangle M)X_f \Lambda_f +(K+\triangle K)X_f\|_{F}=1.2209\times 10^{-14}$ which shows that the no spillover for the unmeasured spectral data is guaranteed. 

Thus we conclude that eigenvalues of $L_{\triangle}(\lam)=\lam^2(M+\triangle M)+(K+\triangle K)$ are $\lam^a_1,\,-\lam^a_1,\,\lam^a_2,\,-\lam^a_2.$ Hence eigenvalues of $L(\lam)$ are replaced by the desired eigenvalues with maintaining no spillover effect.
\end{example}

\subsection{The $\star$-even matrix pencils with positive definite $K$}
Let $L(\lambda)=\lambda M+K$ be a $*$-even matrix pencil with $K>0$.
Then all eigenvalues of $L(\lambda)$ are purely imaginary and there exists a basis 
$x_1^c,\ldots,x_p^c,\, x_{p+1}^f,\ldots , x_{n}^f$
of eigenvectors such that $L(\lambda_i^c)x^c_i=L(\lambda_i^f)x^f_i=0$.
By normalizing the eigenvectors, we may assume
that $(x_i^c)^* Kx^c_j=0$ for $i\not=j$ and $(x_i^c)^* Kx_i^c=1.$ Thus, the $M$-Gramian
of the matrix
$X_c=[x^c_1 \;\ldots\; x^c_p]$ satisfies $G=X_c^*MX_c=-\Lambda_c^{-1}$ as $X_c^*KX_c=I_p,$ where $\Lambda_c={\rm diag}(\lambda^c_i)$, $\Lambda_f={\rm diag}(\lambda^f_i)$, and $\lam_i^c\neq 0$.
Assuming the spectral condition $(a)$ as given in Theorem \ref{theo:stuctmain1} we have
$$\{\lam_1^c, \hdots, \lam_p^c\}\cap \{\lam_{p+1}^f, \hdots, \lam_{n}^f\}=\emptyset$$
fulfilling which the update matrices in Theorem \ref{theo:stuctmain1}
are \begin{equation} \label{DelMDelK_*even}
\Delta M=KX_c\,\hat M X_c^*K, \qquad
\Delta K=KX_c\,(\Lambda_c^{-1}(\Lambda_a-\Lambda_c)-\hat M\Lambda_a) X_c^*K.\end{equation}
Thus $\Delta M=-(\Delta M)^*,\,\Delta K=(\Delta K)^*$ when $\hat M$ and $\Lambda_a$ are diagonal matrices with purely imaginary diagonal entries. 

%If $M$ and $K$ are real matrices then $X_c$ can also be chosen to be real, and consequently the update matrices are also real.

\begin{remark}
If $\hat M=0$ in $(\ref{DelMDelK_*even}),$ then we obtain $\Delta M=0$ and $\Delta K=KX_c\Lambda_c^{-1}(\Lambda_a-\Lambda_c)X_c^*K$ is a Hermitian matrix. On the other hand, assuming $\Lambda_a$ as nonsingular and setting $\hat M=\Lambda_c^{-1}-\Lambda_a^{-1}$, we obtain $\Delta K=0$ and $\Delta M=KX_c(\Lambda_c^{-1}-\Lambda_a^{-1})X_c^*K$ is skew-Hermitian.
\end{remark}
\begin{remark} \label{*evenMK}
If $\{\lam_1^c, \hdots, \lam_p^c\}\cap \{\lam_{p+1}^f, \hdots, \lam_{n}^f\}=\emptyset$ and $\Lambda_a$ is a diagonal matrix with purely imaginary diagonal entries then the structured update matrices in Theorem \ref{theo:stuctmain1}, are given by $\Delta M=KX_c \hat MX_c^*K$ and $\Delta K=KX_c \hat KX_c^*K$ with \begin{equation}
\hat M=H_a\left[\Lambda_c^{-1}(\Lambda_c-\Lambda_a)\Lambda_a+Z_1+Z_2\Lambda_a \right], \,\, \,\, \hat K=H_a\left[\Lambda_c^{-1}(\Lambda_a-\Lambda_c)-Z_1\Lambda_a-Z_2 \Lambda_a^2 \right]
\end{equation} where $H_a=(I_p-\Lambda_a^2)^{-1}$ and $Z_1$ is an arbitrary imaginary diagonal matrix, while $Z_2$ is an arbitrary real diagonal matrix.  

Moreover, if the diagonal matrices $Z_1$ and $Z_2$ are chosen such that $\Lambda_c^{-1}(\Lambda_a-\Lambda_c)-Z_1\Lambda_a-Z_2 \Lambda_a^2$ is a diagonal matrix with non-negative diagonal entries then $\Delta K$ is a positive semi-definite matrix, that is $K+\Delta K>0$.
\end{remark}

Now we consider $T$-even pencils $L(\lam)=\lam M+K\in\R^{n\times n}[\lam]$ where $K>0.$ The structured updates for $L(\lam)$ can be obtained following a similar procedure as described for the case of $T$-odd matrix pencils. Indeed, observe that nonzero complex eigenvalues of $L(\lam)$ are purely imaginary. Let $(\lam_i^c, x_i^c), (\overline{\lam_i^c}, \overline{x_i^c}), 1\leq i\leq p$ be eigenpairs of $L(\lam)$ where the eigenvectors are normalized and $\lam_i^c \neq 0$. Then it may be assumed that $\hat{X}_c^TK\hat{X}_c=I_{2p}$ where $\hat{X}_c=[\re(x_1^c) \, \im(x_2^c) \, \hdots \, \re(x_p^c) \, \im(x_p^c)].$

Then the $M$-Gramian of the matrix $\hat{X}_c$ satisfies $G=\hat{X}^T_cM\hat{X}_c=-\Lambda_c^{-1}$ where $\Lambda_c=\mbox{diag}(\Lambda^c_1,\dots, \Lambda^c_p)$ and $\Lambda_a=\mbox{diag}(\Lambda^a_1,\dots, \Lambda^a_p)$ with $\Lambda^c_j=\bmatrix{0& \im(\lambda^c_j) \\ -\im(\lambda^c_j)& 0}$ and $\Lambda^a_j=\bmatrix{0& \im(\lambda^a_j) \\ -\im(\lambda^a_j)& 0}.$ If the condition $(a)$ of Theorem \ref{theo:stuctmain1} is met, then the update matrices are $$\Delta M=K\hat{X}_c\hat{M}\hat{X}_c^TK,\,\,\,\Delta K=K\hat{X}_c\hat{K}\hat{X}_c^T K$$ where $\hat{M}$ and $\hat{K}$ are solutions of equation (\ref{eq:basiceq2}) in which $X_c$ is replaced by $\hat{X}_c.$

It may also be noted that choosing $\hat{M}=\mbox{diag}(\hat{M}_{11},\dots ,\hat{M}_{pp})$ and $\hat{K}=\Lam_c^{-1}(\Lam_a-\Lam_c) - \hat{M}\Lam_a$ where $\hat{M}_{jj}=\bmatrix{0& \alpha_j \\ -\alpha_j& 0}$ for some $\alpha_1,\hdots, \alpha_p \in \R$, we obtain $\Delta M=-\Delta M^T$ and $\Delta K=\Delta K^T.$ 

\begin{remark}
If the spectral condition $(a)$ in Theorem \ref{theo:stuctmain1} is met, then the structured updates matrices are given by $\Delta M=K\hat{X}_c \hat M\hat{X}_c^TK$ and $\Delta K=K\hat{X}_c \hat K\hat{X}_c^TK$ with \begin{equation}
\hat M=H_a\left[\Lambda_c^{-1}(\Lambda_c-\Lambda_a)\Lambda_a+Z_1+Z_2\Lambda_a \right], \,\, \,\, \hat K=H_a\left[\Lambda_c^{-1}(\Lambda_a-\Lambda_c)-Z_1\Lambda_a-Z_2 \Lambda_a^2 \right]
\end{equation} where $H_a=(I_{2p}-\Lambda_a^2)^{-1}$ and $Z_k=\mbox{diag}(Z^{(k)}_{1},\dots ,Z^{(k)}_{p}),\,k=1,2$ having $Z^{(1)}_{j}=\bmatrix{0& \alpha_j \\ -\alpha_j& 0}$ and $Z^{(2)}_{j}=\beta_j I_2,\, \alpha_j,\,\beta_j \in \R.$ Obviously, $L_\Delta(\lam)=\lam(M+\Delta M)+(K+\Delta K)$ is a $T$-even real matrix pencil with $K+\Delta K>0,$ if $Z_1,\,Z_2$ are chosen such that $\Lambda_c^{-1}(\Lambda_a-\Lambda_c)-Z_1\Lambda_a-Z_2 \Lambda_a^2$ is a diagonal matrix with non-negative diagonal entries.
\end{remark}

Now we consider an example to obtain solution of {\bf (P1)} for undamped models $L(\lam)=\lam^2 M+K$ with $K>0$ by utilizing Remark \ref{*evenMK}. The values of $\lam^2$ to satisfy $\det(\lam^2M+K)=0$ are purely imaginary numbers (not necessary to have self conjugate pair), that is, $\lambda \in \mathcal{E} {\small \setminus} \{0\}.$ Then the eigenvalues of the matrix pencil $L(\lam) = \lam^2 M +K$ occur in pair $(\lam, -\lam)$ corresponding to an eigenvector $x\in \C^n$ for some $\lam \in \mathcal{E} {\small \setminus} \{0\}.$ 

Let $(\pm \lam^c_i,x^c_i),\,i = 1,\dots,p$ denote the eigenpairs of
$L(\lam)=\lam^2 M+K$ and $\pm \lam_i^c$ are to be changed to the aimed eigenvalues $\pm \lam^a_i,\,i = 1,\dots,p$ of $L_{\triangle}(\lam) = \lam^ 2 (M + \triangle M) + (K + \triangle K)$, for some positive semi-definite matrix $\triangle K$ and skew-Hermitian $\triangle M$ with no spillover effect. 
Setting $\Lambda_c=\mbox{diag}\left(\lam^c_1,\lam^c_2,\,\dots\,,\lam^c_p\right)^2,\,\Lambda_a=\mbox{diag}\left(\lam^a_1,\lam^a_2,\,\dots\,,\lam^a_p\right)^2, \, \Lambda_f=\mbox{diag}\left(\lam^f_{p+1},\lam^f_{p+2},\,\dots\,,\lam^f_{n}\right)^2$ and $X_c=\left[x^c_1\,x^c_2\,\dots\,x^c_p\right],$ the MUP with no spillover effect for $L(\lam)$ translates to the problem {\bf{(P1)}}. We consider the following example.

\begin{example}
Suppose $L(\lam)=\lam^2 M+K$ with \begin{center}
$M= \small \bmatrix{
0.00000 + 0.20972i&  -0.10697 + 0.96717i&   0.04080 - 0.91135i & -3.59068 + 1.77061i \\
   0.10697 + 0.96717i&   0.00000 - 0.94422i&  -0.98779 + 1.35265i&   3.55621 - 0.03449i \\
  -0.04080 - 0.91135i&   0.98779 + 1.35265i&   0.00000 - 0.79806i&  -0.50440 - 0.71953i \\
   3.59068 + 1.77061i&  -3.55621 - 0.03449i&   0.50440 - 0.71953i&   0.00000 - 1.82468i },$\\ $K=\footnotesize \bmatrix{5.25927 +  0.00000i&   -1.36185 -  0.39225i&   -1.02993 +  3.85132i&    3.10502 +  0.94912i \\
   -1.36185 +  0.39225i&    5.18883 +  0.00000i&    0.25646 +  2.08573i&    2.82543 -  1.42028i \\
   -1.02993 -  3.85132i&    0.25646 -  2.08573i&   12.57576 +  0.00000i&   -0.35504 -  4.89141i \\
    3.10502 -  0.94912i&    2.82543 +  1.42028i&   -0.35504 +  4.89141i&    9.24337 +  0.00000i}>0.$
\end{center}
Let $\lam^c_1=1.8663(1+i),\,\lam^c_2=0.96032(1+i)$ and $\lam^a_1=1.9538(1+i),\,\lam^a_2=1.1696(1+i).$ Thus we want to replace the eigenvalues $\lam^c_1,\,-\lam^c_1,\,\lam^c_2,\,-\lam^c_2$ of $L(\lam)$ by the desired eigenvalues $\lam^a_1,\,-\lam^a_1,\,\lam^a_2,\,-\lam^a_2$ respectively. So we form 
$\Lambda_c=\mbox{diag}\left(6.96617i,\, 1.84442i\right),\,\Lambda_a=\mbox{diag}\left(7.63484i,\, 2.73573i\right)$ and \begin{center}
$X_c=\bmatrix{0.269248 - 0.049496i&   0.365254 + 0.000000i \\
   0.360869 + 0.000000i&   0.021572 + 0.085644i \\
   0.105515 - 0.042953i&   0.074614 + 0.141519i \\
  -0.030283 + 0.036643i&   0.024397 - 0.220546i}.$ \end{center} 

Therefore by setting $Z_1=\mbox{diag}\left( 0.10025i,\,0.47934i\right)$ and $Z_2=\mbox{diag}\left(0.26054,\,0.84128\right)$ we obtain \begin{center}
$\footnotesize \triangle M= \bmatrix{0.00000 + 0.52241i&  -0.06183 - 0.22791i&   0.00122 - 0.11173i&  -0.41289 + 0.39407i \\
   0.06183 - 0.22791i&  0.00000 + 0.20921i&  -0.13366 + 0.07568i&   0.28312 - 0.05364i \\
  -0.00122 - 0.11173i&   0.13366 + 0.07568i&  0.00000 + 0.17138i&   0.18351 - 0.13284i \\
   0.41289 + 0.39407i&  -0.28312 - 0.05364i&  -0.18351 - 0.13284i&  0.00000 + 0.70160i
},$ \\$\footnotesize\triangle K= \bmatrix{
3.00449 + 0.00000i&  -1.05675 + 0.30905i&  -0.52100 + 0.27793i&   2.41288 + 2.20342i \\
  -1.05675 - 0.30905i&   1.57244 + 0.00000i&   0.52079 + 1.32381i&   0.16989 - 1.66602i \\
  -0.52100 - 0.27793i&   0.52079 - 1.32381i&   1.79857 + 0.00000i&  -0.75754 - 1.70191i \\
   2.41288 - 2.20342i&   0.16989 + 1.66602i&  -0.75754 + 1.70191i&   4.44366 + 0.00000i }\geq 0.$  
\end{center}    

Taking $\Lambda_f=\mbox{diag}(- 5.38777i,\,- 0.38831i)$ and $X_f=\bmatrix{
 -0.129984 - 0.085155i&   0.517601 + 0.000000i \\
  -0.078858 - 0.235812i&   0.286105 - 0.401595i \\
   0.290180 + 0.000000i&   0.036036 + 0.101571i \\
   0.076878 - 0.090772i&  -0.397514 + 0.168257i}$ we obtain
$\|(M+\triangle M)X_f \Lambda_f +(K+\triangle K)X_f\|_{F}=1.8766\times 10^{-14}$ which shows that the no spillover for the unmeasured spectral data is guaranteed. 

Thus we conclude that eigenvalues of $L_{\triangle}(\lam)=\lam^2(M+\triangle M)+(K+\triangle K)$ are $\lam^a_1,\,-\lam^a_1,\,\lam^a_2,\,-\lam^a_2.$ Hence eigenvalues of $L(\lam)$ are replaced by the desired eigenvalues with maintaining no spillover effect.
\end{example}

\section{Updates for $\st$-skew-Hamiltonian/Hamiltonian pencils}
Recall that a matrix pencil $L(\lam)=\lam M+K \in \C^{2n \times 2n}[\lambda]$ is said to be $\st$-skew-Hamiltonian/Hamiltonian (SHH) pencil if $M$ is a $\st$-skew-Hamiltonian matrix and $K$ is a $\st$-Hamiltonian matrix, that is $JM=-(JM)^\st$ and $JK=(JK)^\st$ where $J=\bmatrix{0& I_n \\ -I_n& 0}$ and $\st \in \{*,T\}$ \cite{mehl2000condensed,lin1999canonical}. It is also clear that if $L(\lam)$ is  $\st$-skew-Hamiltonian/Hamiltonian then $JL(\lam)$ is $\st$-even. It is also well-known that if $\lam$ is a simple eigenvalue of $L(\lam)$ with $\re(\lam)\neq 0$ then so is $-\overline{\lam}$, however a purely imaginary eigenvalue need not occur in pairs \cite{benner2002numerical}. Besides, $\lam$ and $-\overline{\lam}$ have the same partial multiplicities \cite{mehl2000condensed}. Our next proposition is about the solution of the problem \textbf{(P2)} for $\st$-SHH pencil $L(\lambda).$ 

\begin{proposition}\label{theo:stuctmain2}
Let $(X_c,\Lambda_c)$ and $(X_f,\Lambda_f)$ be complementary deflating pairs of the pencil $L(\lambda)=\lambda\, M+K$, where $\Lambda_c\in\C^{p\times p}$. 
Suppose that
\begin{center}
$(a)\;$ 
$\sigma(\Lambda_c)\cap 
\sigma(-\Lambda_f^\st)=\emptyset$ \quad and
\quad $(b)\;$ $G:=X_c^\st JMX_c$ is nonsingular.
\end{center}
Let $\Lambda_a,\hat M,\hat K\in\C^{p\times p}$ be such that
\begin{equation}\label{eq:basiceq22}
\hat M \Lambda_a+\hat K=G(\Lambda_c-\Lambda_a).
\end{equation}
Set
$$\Delta M:=J^\st U\hat MU^\st,\qquad \Delta K:=J^\st U\hat K U^\st,
\qquad \text{where }\quad U:=JMX_cG^{-1}.$$
Then $(X_c,\Lambda_a)$ and $(X_f,\Lambda_f)$ are complementary deflating pairs of the pencil $L_\Delta(\lambda)=(M+\Delta M)\, \lambda+(K+\Delta K)$. Furthermore, 
$L_\Delta(\lambda)$ is a SHH pencil whenever 
$\lambda\,\hat M+\hat K   \in \L_p{(\st,-1,1)} $. The latter holds if and only if 
$\lambda\, \hat M+(\hat M+G)\Lambda_a\in
 \L_p{(\st,-1,1)}$.
\end{proposition}
\proof
The proof follows easily from Theorem \ref{theo:stuctmain1}. 
\eproof

%If $L(\lam)=\lam M+K$ is $*$-SHH matrix pencil with $M$ and $K$ are complex then eigenvalues of $L$ exists in pairs $(\lam,-\overline{\lambda})$ whenever $\lam \not \in i\R.$ In particular, if $\lambda$ is an eigenvalue of $L$ with right eigenvector $x$ then $-\overline{\lambda}$ is also an eigenvalue of $L$ corresponding to the left eigenvector $J^Tx$. On the other hand, if $L(\lambda)$ is $T$-SHH matrix pencil then eigenvalues occurs in quadruple $(\lam,-\overline{\lam},\overline{\lambda}, -\lambda)$ whenever $\lambda \not \in \R \cup i\R$ otherwise as $(\lambda,-\lambda).$  

The next result is about the solution of the problem $\textbf{(P1)}$ for $*$-SHH matrix pencil. 

\begin{corollary} \label{thm:Eig_SHH}
Suppose $L(\lam)=\lam M+K$ is a $*$-SHH matrix pencil. Let $(\Lam_c,X_c)$ be a deflating pair of $L(\lam)$ where $\Lambda_c=\mbox{diag}(\lambda^c_1,-\overline{\lam^c_1},\dots,\lambda^c_m,-\overline{\lambda^c_m},\lambda^c_{m+1},\dots,\lambda^c_p),$ $\re(\lam_j^c)\neq 0, j=1,\hdots, m$ and $\lam_k^c, k=m+1,\hdots,p$ are purely imaginary numbers.  Let $\Lambda_a=\mbox{diag}(\lambda^a_1,-\overline{\lam^a_1},\dots,\lambda^a_m,-\overline{\lambda^a_m},\lambda^a_{m+1},\dots,\lambda^a_p)$ where $\re(\lambda^a_j)\neq 0$ and $\lambda^a_k$ are purely imaginary, $j=1,\hdots,m,$ $k=m+1,\hdots,p.$ 

Then by Proposition \ref{theo:stuctmain2}, if it satisfies the conditions $(a),\,(b)$ and $\lam_i^c$s are simple eigenvalues then the update matrices are $\Delta M=J^*U\hat MU^*,\,\, \Delta K=J^*U(G\Lambda_c-(G+\hat{M})\Lambda_a) U^*$ for which $(X_c, \Lam_a), (X_f, \Lam_f)$ are complementary deflating pairs of $L_\Delta(\lam),$ where $\hat{M}$ is an arbitrary matrix of compatible size. This solves problem {\bf (P1)} by unstructured updates.

Further, on choosing $\hat{M}=\mbox{diag}(\hat{M}_{1},\dots ,\hat{M}_{m},\hat{m}_{m+1},\dots, \hat{m}_{p})$ where $\hat{M}_j=\bmatrix{0 & \alpha_j \\ -\overline{\alpha_j} & 0}$, $\re(\alpha_j)\neq 0,$ and $\hat{m}_k$ are purely imaginary numbers, $L_\Delta(\lam)$ becomes $*$-SHH pencil which solves the problem {\bf (P1)} using structured updates.
\end{corollary}
\proof
Since $\lam_i^c$s are simple eigenvalues of $L(\lam),$ the matrix $G$ has the form $G=\mbox{diag}(G_{1},\dots,G_{m},g_{m+1},\dots ,g_{p}) $ where $G_{j}=\bmatrix{0& g_{j}\\-\overline{g}_{j}& 0}$ with $g_j\in \C,$ $1\leq j\leq m$ and $g_{k}, m+1\leq k\leq p$ are imaginary numbers. The rest follows from proposition \ref{theo:stuctmain2}. 
\eproof

Another parametric structured updates are given as follows.

\begin{remark} \label{remk:Eig_SHH}
If the assumptions of Corollary \ref{thm:Eig_SHH} holds then $\Delta M=J^TU\hat MU^*$ and $\Delta K=J^TU\hat{K} U^*$ solves the problem \textbf{(P1)}, where \beano
\hat{M} &=& G(\Lambda_c-\Lambda_a)H_a \Lambda_a^*+Z_1(I_p-\Lambda_aH_a\Lambda_a^*)-Z_2H_a\Lambda_a^*, \\
\hat{K} &=& G(\Lambda_c-\Lambda_a)H_a-Z_1\Lambda_aH_a+Z_2(I_p-H_a)
\eeano
with $H_a=(\Lambda_a^*\Lambda_a+I_p)^{-1},\,G=X_c^*JMX_c,\,U=JMX_cG^{-1}$ and $Z_i=\mbox{diag}(Z^{(i)}_{1},\dots,Z^{(i)}_{m},z^{(i)}_{m+1},\dots ,z^{(i)}_{p}),$ $i=1,2,$ 
$$Z_j^{(1)}=\bmatrix{0 & \alpha_{j} \\ -\overline{\alpha}_{j} & 0}, \, Z_j^{(2)}=\bmatrix{0 & \beta_{j} \\ \overline{\beta}_{j} & 0},\,\alpha_j,\,\beta_j\in \C,\, 1\leq j\leq m$$ and $z_k^{(1)}, m+1\leq k\leq p$ are imaginary numbers and $z_k^{(2)}$ are reals. Besides $\Delta M, \Delta K$ are $*$-skew-Hamiltonian and $*$-Hamiltonian matrix respectively. 
\end{remark}

Now we consider $T$-SHH matrix pencils $L(\lam)=\lam M + K\in \R^{2n\times 2n}[\lam].$ Note that for an eigenvalue $\lam$ of $L(\lam)$ with $\re(\lam)\neq 0\neq \im(\lam),$ $\overline{\lam}, -\overline{\lam}, -\lam$ are also eigenvalues of $L(\lam).$ Moreover, if $x$ and $\hat{x}$ are eigenvectors corresponding to $\lam, -\overline{\lam}$ respectively, then $\overline{x}$ and $\overline{\hat{x}}$ are eigenvectors corresponding to $\overline{\lam}$ and $-\lam$ respectively. If $\im(\lam)=0$ then $\lam, -\lam$ form a pair of eigenvalues of $L(\lam),$ whereas if $\re(\lam)=0$ then $\lam, \overline{\lam}$ are eigenvalues in pairs. Thus for real structured updates of $L(\lam)$ the eigenvalues are to be replaced as tuples depending on the real and imaginary parts of the eigenvalues. Thus we assume that the quadruple of eigenvalues $(\lam_j^c, \overline{\lam_j^c}, -\overline{\lam_j^c}, -\lam_j^c)$ of $L(\lam)$ is to be changed by a quadruple $(\lam_j^a, \overline{\lam_j^a}, -\overline{\lam_j^a}, -\lam_j^a)$ when both the real and imaginary parts of $\lam_j^c$ and $\lam_j^a$ are non zero, where $1\leq j\leq m_1.$ The pair of eigenvalues $(\lam_k^c, \overline{\lam_k^c})$ is to be changed by a pair $(\lam_k^a, \overline{\lam_k^a})$ when the real parts of $\lam_k^c, \lam_k^a$ are zero, $m_1+1 \leq k\leq m_2.$ Finally a pair of eigenvalues $(\lam_l^c, -\lam_l^c)$ of $L(\lam)$ is to be changed by a pair $(\lam_l^a, -\lam_l^a)$ when the imaginary parts of $\lam_l^c, \lam_l^a$ are zero, $m_2+1\leq k\leq p.$ Obviously, $2m_1+2p <n.$

Let  $$X_c = [X^c_1 \dots X^c_{m_1}\,X^c_{m_1+1}\dots X^c_{m_2}\,X^c_{m_2+1} \dots X^c_p]$$ where 
\beano X^c_j &=& [\re(x^c_j)\,\, \im(x^c_j)\,\, \re(\hat{x}^c_j) \,\, \im(\hat{x}^c_j)], \\ X^c_k &=& [\re(x^c_k)\,\, \im(x^c_k)],\\ X^c_l &=& [x^c_l\,\, \hat{x}^c_l],\eeano $x_j^c$ and $\hat{x}_j^c$ denote the eigenvectors corresponding to $\lam_j^c$ and $-\overline{\lam_j^c}$ respectively, and $x_k^c,$ $x_l^c$ and $\hat{x}_l^c$ denote the eigenvectors corresponding to the eigenvalues $\lam_k^c,$ $\lam_l^c$ and $-\lam_l^c$ respectively. 

Further, suppose 
\beano
\Lambda_c &=& \mbox{diag}(\Lambda^c_1,\dots ,\Lambda^c_{m_1},\Lambda^c_{m_1+1},\dots ,\Lambda^c_{m_2},\,\Lambda^c_{m_2+1},\dots, \Lambda^c_p) \\
\Lambda_a &=& \mbox{diag}(\Lambda^a_1,\dots ,\Lambda^a_{m_1},\Lambda^a_{m_1+1},\dots ,\Lambda^a_{m_2},\,\Lambda^a_{m_2+1},\dots, \Lambda^a_p)
\eeano
where \beano && \Lambda^c_j=\mbox{diag}(\hat{\Lambda}^c_{j},\,-(\hat{\Lambda}^c_{j})^T),\,\Lambda^c_k=\bmatrix{0& \im(\lambda^c_k)\\ - \im(\lambda^c_k)& 0},\,\Lambda^c_l=\mbox{diag}(\lambda^c_l,\,-\lambda^c_l),\\ 
&& \Lambda^a_j=\mbox{diag}(\hat{\Lambda}^a_{j},\,-(\hat{\Lambda}^a_{j})^T),\,\Lambda^a_k=\bmatrix{0& \im(\lambda^a_k)\\ - \im(\lambda^a_k)& 0},\,\Lambda^a_l=\mbox{diag}(\lambda^a_l,\,-\lambda^a_l)\eeano and $\hat{\Lambda}^c_{j}=\bmatrix{\re(\lambda^c_j)& \im(\lambda^c_j) \\ - \im(\lambda^c_j)& \re(\lambda^c_j)},\,\hat{\Lambda}^a_{j}=\bmatrix{\re(\lambda^a_j)& \im(\lambda^a_j) \\- \im(\lambda^a_j)& \re(\lambda^a_j)},\,j=1, \hdots, m_1,\,k=m_1+1,\hdots,m_2,\,l=m_2+1,\hdots,p.$

Then we have the following theorem.

\begin{theorem} \label{thm:Eig_TSHH}
Let $(X_c, \Lam_c)$ be the eigenpair matrix of the $T$-SHH matrix pencil $L(\lam)=\lam M+K$ as described above. Then by Proposition \ref{theo:stuctmain2}, if it satisfies the conditions $(a),\,(b)$ and all the to be changed eigenvalues are distinct then the update matrices are $\Delta M=J^TU\hat MU^T,\,\, \Delta K=J^TU(G\Lambda_c-(G+\hat{M})\Lambda_a) U^T,$ where $\hat{M}$ is an arbitrary matrix of compatible size.

In addition, choosing $\hat{M}=\mbox{diag}(\hat{M}_{1},\dots ,\hat{M}_{m_1},\hat{M}_{m_1+1},\dots, \hat{M}_{m_2},\hat{M}_{m_2+1} ,\dots ,\hat{M}_{p})$ we obtain $T$-skew-Hamiltonian $\Delta M$ and $T$-Hamiltonian $\Delta K$ which solves the problem \textbf{(P1)} where $$\hat{M}_{j}=\bmatrix{{\bf 0} & \alpha_j I_2+\beta_j J_2\\ -\alpha_j I_2+\beta_j J_2& {\bf 0}},\,\hat{M}_{k}=\beta_kJ_2,\,\hat{M}_{l}=\beta_lJ_2,$$ ${\bf 0}$ is the zero matrix, $J_2=\bmatrix{0& 1\\-1& 0},\,\alpha_j,\,\beta_j,\,\beta_k,\,\beta_l$ are arbitrary real numbers and $j=1,\hdots, m_1,\,k=m_1+1,\hdots, m_2,\,l=m_2+1,\hdots,p.$ 
\end{theorem}
\proof
As the eigenvalues of $\Lambda_c$ are distinct so the matrix $G=X^T_cJMX_c$ is of the form  $G=\mbox{diag}(G_{1},\dots,G_{m_1},G_{m_1+1},\dots ,G_{m_2},\,G_{m_2+1},\dots ,G_{p}) $ where $G_{j}=\bmatrix{0_2& u_jI_2+v_jJ_2\\ -u_jI_2+v_jJ_2& 0_2},$ $G_{k}=v_kJ_2,\,G_{l}=v_lJ_2$ for some real numbers $u_j,\,v_j,\,v_k,\,v_l$, $j=1,\hdots,m_1,\,k=m_1+1,\hdots, m_2,\,l=m_2+1,\hdots,p.$ Rest of the proof follows from Proposition \ref{theo:stuctmain2}. 
\eproof

Another parametric updates $\Delta M, \Delta K$ which solves the problem {\bf (P1)} for $T$-SHH pencils can be represented as follows.

\begin{remark} \label{remk:Eig_TSHH}
If the assumptions of Theorem \ref{thm:Eig_TSHH} hold then $T$-skew-Hamiltonian update matrix $\Delta M=J^TU\hat MU^T$ and  $T$-Hamiltonian matrix is given by $\Delta K=J^TU\hat{K} U^T$ which solves the problem \textbf{(P1)}, where
\beano
\hat{M} &=& G(\Lambda_c-\Lambda_a)H_a \Lambda_a^T+Z_1(I_{2m_1+2p}-\Lambda_aH_a\Lambda_a^T)-Z_2H_a\Lambda_a^T, \\
\hat{K} &=& G(\Lambda_c-\Lambda_a)H_a-Z_1\Lambda_aH_a+Z_2(I_{2m_1+2p}-H_a)
\eeano
with $H_a=(\Lambda_a^T\Lambda_a+I_{2m_1+2p})^{-1},\,G=X_c^TJMX_c,\,U=JMX_cG^{-1},$ $$Z_i=\mbox{diag}(Z^{(i)}_{1},\dots,Z^{(i)}_{m_1},Z^{(i)}_{m_1+1},\dots ,Z^{(i)}_{m_2},\,Z^{(i)}_{m_2+1},\dots ,Z^{(i)}_{p}),i=1,2,$$ $Z^{(1)}_{j}=\bmatrix{\textbf{0}& \alpha_{j}I_2+\beta_{j} J_2\\-\alpha_{j}I_2+\beta_{j} J_2& \textbf{0}},\,Z^{(2)}_{j}=\bmatrix{\textbf{0}& u_{j}I_2+v_{j} J_2\\ u_{j}I_2-v_{j} J_2& \textbf{0}},\,Z^{(1)}_{k}=\beta_{k} J_2,\, Z^{(2)}_{k}=u_{k}I_2,\,Z^{(1)}_{l}=\beta_{l} J_2,\,Z^{(2)}_{l}=u_{l} \bmatrix{0& 1\\1& 0},$ and $\alpha_{j},\,\beta_{j},u_{j},\,v_{j},\,\beta_{k},\,u_{k},\, \beta_{l},\,u_{l}$ are arbitrary real numbers, $j=1, \hdots,m_1,\,k=m_1+1,\hdots,m_2,\,l=m_2+1,\hdots,p.$
\end{remark}

Now we apply the above results on a numerical example to examine the validity of the results. 
\begin{example}
Consider a $*$-SHH pencil $L(\lam)=\lam M+K$ with 
$$M=\bmatrix{-0.25455 + 0.95256i&   0.02934 + 0.05513i&   0.00000 - 1.83635i&   0.08681 - 1.45077i \\
   2.25023 - 0.01156i&   1.14852 - 1.53017i&  -0.08681 - 1.45077i&   0.00000 + 1.40120i \\
   0.00000 - 0.96582i&  -0.22366 - 0.46730i&  -0.25455 - 0.95256i &  2.25023 + 0.01156i \\
   0.22366 - 0.46730i&   0.00000 - 1.00248i&   0.02934 - 0.05513i&   1.14852 + 1.53017i},$$ $$K=\bmatrix{
3.02148 + 1.90489i&   1.10499 + 1.16245i&  -1.26366 + 0.00000i&   1.65942 + 0.71011i \\
   0.44232 - 1.07299i&   0.29350 - 0.24688i&   1.65942 - 0.71011i&  -0.19304 + 0.00000i \\
   1.30628 + 0.00000i &  -0.42739 + 0.75761i&  -3.02148 + 1.90489i&  -0.44232 - 1.07299i \\
  -0.42739 - 0.75761i&   0.52491 + 0.00000i&  -1.10499 + 1.16245i&  -0.29350 - 0.24688i}.$$

Let $\lam^c_1=-0.92332 - 0.75639i,\,\lam^c_2=-0.12114i$ and $\lam^a_1=-0.76954+0.53243i,\,\lam^a_2=-3.22147i.$ Suppose that we want to replace the set of eigenvalues $\{\lam^c_1,\,-\overline{\lam^c_1},\,\lam^c_2 \}$ of $L(\lam)$ by the desired set of eigenvalues $\{\lam^a_1,\,-\overline{\lam^a_1},\,\lam^a_2 \}$ respectively. Thus
$\Lambda_c=\mbox{diag}\left(\lam^c_1,\,-\overline{\lam^c_1},\,\lam^c_2 \right),\,\Lam_a=\mbox{diag}\left(\lam^a_1,\,-\overline{\lam^a_1},\,\lam^a_2\right)$ and \begin{center}$X_c=\bmatrix{
 1.00000 + 0.00000i&  -0.43182 + 0.23755i&  -0.20930 + 0.22721i \\
-0.32603 - 0.60175i&   1.00000 + 0.00000i&  -0.67852 - 0.58802i \\
 0.72475 + 0.50622i&  -0.01383 + 0.37218i&   0.21160 - 0.29125i \\
-0.20761 + 0.69892i&   0.09784 + 0.45636i&   1.00000 + 0.00000i}.$ \end{center} 

Then by remark \ref{remk:Eig_SHH}, choosing $Z_1=\bmatrix{0&   0.06022 + 0.19082i&   0 \\
  -0.06022 + 0.19082i&   0&   0 \\
   0&   0&   1.19827i},\\ Z_2=\bmatrix{0&  -0.50561 + 0.37741i&   0 \\
  -0.50561 - 0.37741i&   0&   0 \\
   0&  0&   1.45556}$ we obtain $$\triangle M= \bmatrix{
0.27615 + 0.21015i&  -0.64643 - 1.17676i&   0.00000 - 0.45391i&   0.95858 + 0.57857i \\
  -0.88139 - 0.13297i&  -1.84854 + 0.99750i&  -0.95858 + 0.57857i&   0.00000 - 2.19806i \\
  -0.00000 + 0.70112i&   0.64985 - 0.15198i&   0.27615 - 0.21015i&  -0.88139 + 0.13297i \\
  -0.64985 - 0.15198i&   0.00000 + 1.69525i&  -0.64643 + 1.17676i&  -1.84854 - 0.99750i
},$$ $$\triangle K=\bmatrix{
-0.63477 - 1.42656i&  -1.93590 - 0.08067i&  -2.43388 + 0.00000i&   0.04977 - 2.40635i \\
  -1.43606 + 0.85246i&   0.29333 + 1.96152i&   0.04977 + 2.40635i&  -2.93978 + 0.00000i \\
   0.86197 - 0.00000i&   0.63350 - 1.45810i&   0.63477 - 1.42656i&   1.43606 + 0.85246i \\
   0.63350 + 1.45810i&   1.46857 - 0.00000i&   1.93590 - 0.08067i&  -0.29333 + 1.96152i
 }.$$ 
On taking $\Lambda_f=4.51104i$ and $X_f=\bmatrix{
0.20548 + 0.72300i\\
  -0.52204 + 0.39798i \\
   1.00000 - 0.00000i \\
  -0.61073 + 0.21633i}$ we obtain
$\|(M+\triangle M)X_f \Lambda_f+(K+\triangle K)X_f \|_F=1.5519 \times 10^{-14},\,$ which shows that the unmeasured spectral data remain undisturbed. 

Hence we conclude that eigenvalues of the 
$*$-SHH pencil $L_{\triangle}(\lam) = \lam (M+\triangle M)+(K+\triangle K)$ are $\{\lam^a_1,\,-\overline{\lam^a_1},\,\lam^a_2 \}.$ Therefore eigenvalues of $L(\lam)$ are replaced by the desired eigenvalues with maintaining no spillover condition. 

\end{example}

{\bf Conclusion} Given a matrix pencil $L(\lam)=\lam M +K\in\C^{n\times n}[\lam],$ a matrix pair $(X,\Lambda)\in \C^{n\times p} \times \C^{p\times p}$ is said to be a deflating pair of $L(\lam)$ if $MX\Lambda+KX=0,$ $p< n.$ Two such deflating pairs $(X_1, \Lam_1)\in \C^{n\times p} \times \C^{p\times p}$ and $(X_2, \Lam_2)\in \C^{n\times (n-p)} \times \C^{(n-p)\times (n-p)}$ are called complementary if $[X_1 \,\, X_2]$ is invertible. Given the complementary deflating pairs $(X_c,\Lam_c)$ and $(X_f,\Lam_f)$ of a structured matrix pencil $L(\lam),$ and an another matrix pair $(X_a,\Lam_a)$ we determine computable expressions of structured and unstructured updates $\Delta M, \Delta K$ such that the updated matrix pencil $L_\Delta(\lam)=\lam (M+\Delta M) + (K+\Delta K)$ inherit $(X_a,\Lam_a), (X_f,\Lam_f)$ as complementary deflating pairs under some generic assumptions. When the matrices $\Lam_c, \Lam_f$ and $\Lam_a$ are diagonal matrices then the above problem is called the model updating problem with no spillover, in which the diagonal entries of $\Lam_a$ and $\Lam_f$ are the measured and unmeasured eigenvalues of a undamped finite element model associated with the pencil $L(\lam).$ However, in general $(X_f,\Lam_f)$ is not known and with this assumption we derive explicit parametric expression of unstructured and structured updates for a variety of structured matrix pencils which include symmetric, Hermitian, $\st$-even, $\st$-odd and $\st$-skew-Hamiltonian/Hamiltonian matrix pencils. We examine the validity of the theoretical results by considering several numerical examples. We plan to extend the proposed framework to finite element quadratic model updating problem with no spillover. \\\\

\noin{\bf Acknowledgment} Biswa Nath Datta acknowledges IIT Kharagpur for providing necessary support for his several visits to IIT Kharagpur. Michael Karow acknowledges IIT Kharagpur for the support through SGRIP grant and GIAN course which made his visit to IIT Kharagpur possible. 

%Given a structured matrix pencil $L(\lam)=\lam M+K $ we derive explicit parametric expressions of the perturbations $\triangle M, \triangle K$ of $M, K$ respectively which solve the model updating problem with no spillover effect on the spectral data where $L(\lam)$ is symmetric, Hermitian, $T$-odd or $*$-odd, \textcolor{blue}{$T$-even or $*$-even and $T$-SHH or $*$-Skew-Hamiltonian/Hamiltonian pencils}. Besides, we determine structured perturbations $\triangle M, \triangle K$ such that the updated pencil has the same structure as the original one.
% Then we solve the model updating problem with no spillover 
%effect on deflating subspace of unstructured matrix pencils
% and structured matrix pencils as said above. This problem 
% can be considered as a generalization of the previous problem.

%\bibliographystyle{plain}
%\bibliography{referen}
\end{document}